\documentclass[11pt,a4paper]{article}
\usepackage[UKenglish]{babel}
\usepackage[a4paper,margin=1in]{geometry}
\usepackage{adjustbox}
\usepackage{amsthm}
\usepackage{amsmath}
\usepackage{amssymb}
\usepackage{physics}
\usepackage{stmaryrd}
\usepackage{tikz}
\usepackage{tikz-cd}
\usepackage{tikzit}
\usepackage{circuitikz}
\usepackage{hyperref}
\hypersetup{colorlinks=true, unicode=true, linkcolor=[rgb]{0.10,0.05,0.67}, citecolor=[rgb]{0.10,0.05,0.67}, filecolor=[rgb]{0.10,0.05,0.67}, urlcolor=[rgb]{0.10,0.05,0.67}}
\usepackage{tubes}
\usepackage{xsavebox}
\usepackage{cmll}

\tikzcdset{scale cd/.style={every label/.append style={scale=#1},
    cells={nodes={scale=#1}}}}

\usetikzlibrary{shapes.multipart}

\tikzstyle{bwSpider}=[
       rectangle split,
       rectangle split parts=2,
       rectangle split part fill={black,white},
 minimum size=3.6 mm, inner sep=-2mm, draw=black,scale=0.5,rounded corners=0.8 mm
       ]
 \tikzstyle{wbSpider}=[
       rectangle split,
       rectangle split parts=2,
       rectangle split part fill={white,black},
 minimum size=3.6 mm, inner sep=-2mm, draw=black,scale=0.5,rounded corners=0.8 mm
       ]
\tikzstyle{cWire}=[densely dotted, thick]
\tikzstyle{env}=[copoint,regular polygon rotate=0,minimum width=0.2cm, fill=black]

\tikzstyle{probs}=[shape=semicircle,fill=white,draw=black,shape border rotate=180,minimum width=1.2cm]

\tikzstyle{every picture}=[baseline=-0.25em,scale=0.5]
\tikzstyle{dotpic}=[] 
\tikzstyle{diredges}=[every to/.style={diredge}]
\tikzstyle{math matrix}=[matrix of math nodes,left delimiter=(,right delimiter=),inner sep=2pt,column sep=1em,row sep=0.5em,nodes={inner sep=0pt},text height=1.5ex, text depth=0.25ex]

\tikzstyle{inline text}=[text height=1.5ex, text depth=0.25ex,yshift=0.5mm]
\tikzstyle{label}=[font=\footnotesize,text height=1.5ex, text depth=0.25ex,yshift=0.5mm]
\tikzstyle{left label}=[label,anchor=east,xshift=1.5mm]
\tikzstyle{right label}=[label,anchor=west,xshift=-1.5mm]

\tikzstyle{braceedge}=[decorate,decoration={brace,amplitude=2mm,raise=-1mm}]
\tikzstyle{small braceedge}=[decorate,decoration={brace,amplitude=1mm,raise=-1mm}]

\tikzstyle{doubled}=[line width=1.6pt] 
\tikzstyle{boldedge}=[doubled,shorten <=-0.17mm,shorten >=-0.17mm]
\tikzstyle{boldedgegray}=[doubled,gray,shorten <=-0.17mm,shorten >=-0.17mm]
\tikzstyle{singleedgegray}=[gray]

\tikzstyle{semidoubled}=[line width=1.4pt] 
\tikzstyle{semiboldedgegray}=[semidoubled,gray,shorten <=-0.17mm,shorten >=-0.17mm]

\tikzstyle{boxedge}=[semiboldedgegray]

\tikzstyle{dottededge}=[dashed,shorten <=-0.02mm,shorten >=-0.02mm]

\tikzstyle{boldedgedashed}=[very thick,dashed,shorten <=-0.17mm,shorten >=-0.17mm]
\tikzstyle{vboldedgedashed}=[doubled,dashed,shorten <=-0.17mm,shorten >=-0.17mm]
\tikzstyle{left hook arrow}=[left hook-latex]
\tikzstyle{right hook arrow}=[right hook-latex]
\tikzstyle{sembracket}=[line width=0.5pt,shorten <=-0.07mm,shorten >=-0.07mm]

\tikzstyle{causal edge}=[->,thick,gray]
\tikzstyle{causal nondir}=[thick,gray]
\tikzstyle{timeline}=[thick,gray, dashed]

\tikzstyle{cedge}=[<->,thick,gray!70!white]

\tikzstyle{empty diagram}=[draw=gray!40!white,dashed,shape=rectangle,minimum width=1cm,minimum height=1cm]
\tikzstyle{empty diagram small}=[draw=gray!50!white,dashed,shape=rectangle,minimum width=0.6cm,minimum height=0.5cm]

\tikzstyle{dot}=[inner sep=0mm,minimum width=2mm,minimum height=2mm,draw,shape=circle]

\tikzstyle{phase dot}=[pdot,phase dimensions]
\tikzstyle{wphase dot}=[dot, phase dimensions]

\tikzstyle{leak}=[white dot, shape=regular polygon, minimum size=300mm, regular polygon sides=3, outer sep=-0.2mm, regular polygon rotate=270]
\tikzstyle{preleak}=[trapezium, trapezium angle=67.5, draw, inner sep=0.1pt, outer sep=0pt, minimum height=2mm, minimum width=4pt,rotate=270]
\tikzstyle{proj}=[white dot, shape=regular polygon, minimum size=3.3 mm, regular polygon sides=4, outer sep=-0.2mm]
\tikzstyle{Vleak}=[white dot, shape=regular polygon, minimum size=3.3 mm, regular polygon sides=3, outer sep=-0.2mm, regular polygon rotate=90]
\tikzstyle{dleak}=[white dot, line width=1.6pt, shape=regular polygon, minimum size=3.3 mm, regular polygon sides=3, outer sep=-0.2mm, regular polygon rotate=270]

\tikzstyle{Wsquare}=[white dot, shape=regular polygon, rounded corners=0.8 mm, minimum size=3.3 mm, regular polygon sides=3, outer sep=-0.2mm]
\tikzstyle{Wsquareadj}=[white dot, shape=regular polygon, rounded corners=0.8 mm, minimum size=3.3 mm, regular polygon sides=3, outer sep=-0.2mm, regular polygon rotate=180]
\tikzstyle{ddot}=[inner sep=0mm, doubled, minimum width=2.5mm,minimum height=2.5mm,draw,shape=circle]

\tikzstyle{black dot}=[dot,fill=black]
\tikzstyle{white dot}=[dot,fill=white,,text depth=-0.2mm]
\tikzstyle{white Wsquare}=[Wsquare,fill=gray,,text depth=-0.2mm]
\tikzstyle{white Wsquareadj}=[Wsquareadj,fill=white,,text depth=-0.2mm]
\tikzstyle{green dot}=[white dot] 
\tikzstyle{gray dot}=[dot,fill=gray!40!white,,text depth=-0.2mm]
\tikzstyle{red dot}=[gray dot] 

\tikzstyle{black ddot}=[ddot,fill=black]
\tikzstyle{white ddot}=[ddot,fill=white]
\tikzstyle{gray ddot}=[ddot,fill=gray!40!white]

\tikzstyle{gray edge}=[gray!60!white]

\tikzstyle{small dot}=[inner sep=0.5mm,minimum width=0pt,minimum height=0pt,draw,shape=circle]

\tikzstyle{small black dot}=[small dot,fill=black]
\tikzstyle{small white dot}=[small dot,fill=white]
\tikzstyle{small gray dot}=[small dot,fill=gray!40!white]

\tikzstyle{causal dot}=[inner sep=0.4mm,minimum width=0pt,minimum height=0pt,draw=white,shape=circle,fill=gray!40!white]

\tikzstyle{phase dimensions}=[minimum size=5mm,font=\footnotesize,rectangle,rounded corners=2.5mm,inner sep=0.2mm,outer sep=-2mm]
\tikzstyle{dphase dimensions}=[minimum size=5mm,font=\footnotesize,rectangle,rounded corners=2.5mm,inner sep=0.2mm,outer sep=-2mm]

\tikzstyle{white phase dot}=[dot,fill=white,phase dimensions]
\tikzstyle{white phase ddot}=[ddot,fill=white,dphase dimensions]

\tikzstyle{white rect ddot}=[draw=black,fill=white,doubled,minimum size=5mm,font=\footnotesize,rectangle,rounded corners=2.5mm,inner sep=0.2mm]
\tikzstyle{gray rect ddot}=[draw=black,fill=gray!40!white,doubled,minimum size=6mm,font=\footnotesize,rectangle,rounded corners=3mm]

\tikzstyle{gray phase dot}=[dot,fill=gray!40!white,phase dimensions]
\tikzstyle{gray phase ddot}=[ddot,fill=gray!40!white,dphase dimensions]
\tikzstyle{grey phase dot}=[gray phase dot]
\tikzstyle{grey phase ddot}=[gray phase ddot]

\tikzstyle{small phase dimensions}=[minimum size=4mm,font=\tiny,rectangle,rounded corners=2mm,inner sep=0.2mm,outer sep=-2mm]
\tikzstyle{small dphase dimensions}=[minimum size=4mm,font=\tiny,rectangle,rounded corners=2mm,inner sep=0.2mm,outer sep=-2mm]

\tikzstyle{small gray phase dot}=[dot,fill=gray!40!white,small phase dimensions]
\tikzstyle{small gray phase ddot}=[ddot,fill=gray!40!white,small dphase dimensions]

\tikzstyle{small map}=[draw,shape=rectangle,minimum height=4mm,minimum width=4mm,fill=white]

\tikzstyle{cnot}=[fill=white,shape=circle,inner sep=-1.4pt]

\tikzstyle{asym hadamard}=[fill=white,draw,shape=NEbox,inner sep=0.6mm,font=\footnotesize,minimum height=4mm]
\tikzstyle{asym hadamard conj}=[fill=white,draw,shape=NWbox,inner sep=0.6mm,font=\footnotesize,minimum height=4mm]
\tikzstyle{asym hadamard dag}=[fill=white,draw,shape=SEbox,inner sep=0.6mm,font=\footnotesize,minimum height=4mm]

\tikzstyle{hadamard}=[fill=white,draw,inner sep=0.6mm,font=\footnotesize,minimum height=4mm,minimum width=4mm]
\tikzstyle{small hadamard}=[fill=white,draw,inner sep=0.6mm,minimum height=1.5mm,minimum width=1.5mm]
\tikzstyle{small hadamard rotate}=[small hadamard,rotate=45]
\tikzstyle{dhadamard}=[hadamard,doubled]
\tikzstyle{small dhadamard}=[small hadamard,doubled]
\tikzstyle{small dhadamard rotate}=[small hadamard rotate,doubled]
\tikzstyle{antipode}=[white dot,inner sep=0.3mm,font=\footnotesize]

\tikzstyle{scalar}=[diamond,draw,inner sep=0.5pt,font=\small]
\tikzstyle{dscalar}=[diamond,doubled, draw,inner sep=0.5pt,font=\small]

\tikzstyle{small box}=[rectangle,inline text,fill=white,draw,minimum height=5mm,yshift=-0.5mm,minimum width=5mm,font=\small]
\tikzstyle{small gray box}=[small box,fill=gray!30]
\tikzstyle{medium box}=[rectangle,inline text,fill=white,draw,minimum height=5mm,yshift=-0.5mm,minimum width=10mm,font=\small]
\tikzstyle{square box}=[small box] 
\tikzstyle{medium gray box}=[small box,fill=gray!30]
\tikzstyle{semilarge box}=[rectangle,inline text,fill=white,draw,minimum height=5mm,yshift=-0.5mm,minimum width=12.5mm,font=\small]
\tikzstyle{large box}=[rectangle,inline text,fill=white,draw,minimum height=5mm,yshift=-0.5mm,minimum width=15mm,font=\small]
\tikzstyle{large gray box}=[small box,fill=gray!30]

\tikzstyle{Bayes box}=[rectangle,fill=black,draw, minimum height=3mm, minimum width=3mm]

\tikzstyle{gray square point}=[small box,fill=gray!50]

\tikzstyle{dphase box white}=[dhadamard]
\tikzstyle{dphase box gray}=[dhadamard,fill=gray!50!white]
\tikzstyle{phase box white}=[hadamard]
\tikzstyle{phase box gray}=[hadamard,fill=gray!50!white]

\tikzstyle{point}=[regular polygon,regular polygon sides=3,draw,scale=0.75,inner sep=-0.5pt,minimum width=9mm,fill=white,regular polygon rotate=180]
\tikzstyle{point nosep}=[regular polygon,regular polygon sides=3,draw,scale=0.75,inner sep=-2pt,minimum width=9mm,fill=white,regular polygon rotate=180]
\tikzstyle{copoint}=[regular polygon,regular polygon sides=3,draw,scale=0.75,inner sep=-0.5pt,minimum width=9mm,fill=white]
\tikzstyle{dpoint}=[point,doubled]
\tikzstyle{dcopoint}=[copoint,doubled]

\tikzstyle{pointgrow}=[shape=cornerpoint,kpoint common,scale=0.75,inner sep=3pt]
\tikzstyle{pointgrow dag}=[shape=cornercopoint,kpoint common,scale=0.75,inner sep=3pt]

\tikzstyle{wide copoint}=[fill=white,draw,shape=isosceles triangle,shape border rotate=90,isosceles triangle stretches=true,inner sep=0pt,minimum width=1.5cm,minimum height=6.12mm]
\tikzstyle{wide point}=[fill=white,draw,shape=isosceles triangle,shape border rotate=-90,isosceles triangle stretches=true,inner sep=0pt,minimum width=1.5cm,minimum height=6.12mm,yshift=-0.0mm]
\tikzstyle{wide point plus}=[fill=white,draw,shape=isosceles triangle,shape border rotate=-90,isosceles triangle stretches=true,inner sep=0pt,minimum width=1.74cm,minimum height=7mm,yshift=-0.0mm]

\tikzstyle{wide dpoint}=[fill=white,doubled,draw,shape=isosceles triangle,shape border rotate=-90,isosceles triangle stretches=true,inner sep=0pt,minimum width=1.5cm,minimum height=6.12mm,yshift=-0.0mm]

\tikzstyle{tinypoint}=[regular polygon,regular polygon sides=3,draw,scale=0.55,inner sep=-0.15pt,minimum width=6mm,fill=white,regular polygon rotate=180]

\tikzstyle{white point}=[point]
\tikzstyle{white dpoint}=[dpoint]
\tikzstyle{green point}=[white point] 
\tikzstyle{white copoint}=[copoint]
\tikzstyle{gray point}=[point,fill=gray!40!white]
\tikzstyle{gray dpoint}=[gray point,doubled]
\tikzstyle{red point}=[gray point] 
\tikzstyle{gray copoint}=[copoint,fill=gray!40!white]
\tikzstyle{gray dcopoint}=[gray copoint,doubled]

\tikzstyle{white point guide}=[regular polygon,regular polygon sides=3,font=\scriptsize,draw,scale=0.65,inner sep=-0.5pt,minimum width=9mm,fill=white,regular polygon rotate=180]

\tikzstyle{black point}=[point,fill=black,font=\color{white}]
\tikzstyle{black copoint}=[copoint,fill=black,font=\color{white}]

\tikzstyle{tiny gray point}=[tinypoint,fill=gray!40!white]

\tikzstyle{diredge}=[->]
\tikzstyle{ddiredge}=[<->]
\tikzstyle{rdiredge}=[<-]
\tikzstyle{thickdiredge}=[->, very thick]
\tikzstyle{pointer edge}=[->,very thick,gray]
\tikzstyle{pointer edge part}=[very thick,gray]
\tikzstyle{dashed edge}=[dashed]
\tikzstyle{thick dashed edge}=[very thick,dashed]
\tikzstyle{thick gray dashed edge}=[thick dashed edge,gray!40]
\tikzstyle{thick map edge}=[very thick,|->]

\makeatletter
\newcommand{\boxshape}[3]{%
\pgfdeclareshape{#1}{
\inheritsavedanchors[from=rectangle] 
\inheritanchorborder[from=rectangle]
\inheritanchor[from=rectangle]{center}
\inheritanchor[from=rectangle]{north}
\inheritanchor[from=rectangle]{south}
\inheritanchor[from=rectangle]{west}
\inheritanchor[from=rectangle]{east}
\backgroundpath{
\southwest \pgf@xa=\pgf@x \pgf@ya=\pgf@y
\northeast \pgf@xb=\pgf@x \pgf@yb=\pgf@y

\@tempdima=#2
\@tempdimb=#3

\pgfpathmoveto{\pgfpoint{\pgf@xa - 5pt + \@tempdima}{\pgf@ya}}
\pgfpathlineto{\pgfpoint{\pgf@xa - 5pt - \@tempdima}{\pgf@yb}}
\pgfpathlineto{\pgfpoint{\pgf@xb + 5pt + \@tempdimb}{\pgf@yb}}
\pgfpathlineto{\pgfpoint{\pgf@xb + 5pt - \@tempdimb}{\pgf@ya}}
\pgfpathlineto{\pgfpoint{\pgf@xa - 5pt + \@tempdima}{\pgf@ya}}
\pgfpathclose
}
}}

\boxshape{NEbox}{0pt}{5pt}
\boxshape{SEbox}{0pt}{-5pt}
\boxshape{NWbox}{5pt}{0pt}
\boxshape{SWbox}{-5pt}{0pt}
\boxshape{EBox}{-3pt}{3pt}
\boxshape{WBox}{3pt}{-3pt}
\makeatother

\tikzstyle{cloud}=[shape=cloud,draw,minimum width=1.5cm,minimum height=1.5cm]

\tikzstyle{map}=[draw,shape=NEbox,inner sep=2pt,minimum height=6mm,fill=white]
\tikzstyle{dashedmap}=[draw,dashed,shape=NEbox,inner sep=2pt,minimum height=6mm,fill=white]
\tikzstyle{mapdag}=[draw,shape=SEbox,inner sep=2pt,minimum height=6mm,fill=white]
\tikzstyle{mapadj}=[draw,shape=SEbox,inner sep=2pt,minimum height=6mm,fill=white]
\tikzstyle{maptrans}=[draw,shape=SWbox,inner sep=2pt,minimum height=6mm,fill=white]
\tikzstyle{mapconj}=[draw,shape=NWbox,inner sep=2pt,minimum height=6mm,fill=white]

\tikzstyle{medium map}=[draw,shape=NEbox,inner sep=2pt,minimum height=6mm,fill=white,minimum width=7mm]
\tikzstyle{medium map dag}=[draw,shape=SEbox,inner sep=2pt,minimum height=6mm,fill=white,minimum width=7mm]
\tikzstyle{medium map adj}=[draw,shape=SEbox,inner sep=2pt,minimum height=6mm,fill=white,minimum width=7mm]
\tikzstyle{medium map trans}=[draw,shape=SWbox,inner sep=2pt,minimum height=6mm,fill=white,minimum width=7mm]
\tikzstyle{medium map conj}=[draw,shape=NWbox,inner sep=2pt,minimum height=6mm,fill=white,minimum width=7mm]
\tikzstyle{semilarge map}=[draw,shape=NEbox,inner sep=2pt,minimum height=6mm,fill=white,minimum width=9.5mm]
\tikzstyle{semilarge map trans}=[draw,shape=SWbox,inner sep=2pt,minimum height=6mm,fill=white,minimum width=9.5mm]
\tikzstyle{semilarge map adj}=[draw,shape=SEbox,inner sep=2pt,minimum height=6mm,fill=white,minimum width=9.5mm]
\tikzstyle{semilarge map dag}=[draw,shape=SEbox,inner sep=2pt,minimum height=6mm,fill=white,minimum width=9.5mm]
\tikzstyle{semilarge map conj}=[draw,shape=NWbox,inner sep=2pt,minimum height=6mm,fill=white,minimum width=9.5mm]
\tikzstyle{large map}=[draw,shape=NEbox,inner sep=2pt,minimum height=6mm,fill=white,minimum width=12mm]
\tikzstyle{large map conj}=[draw,shape=NWbox,inner sep=2pt,minimum height=6mm,fill=white,minimum width=12mm]
\tikzstyle{very large map}=[draw,shape=NEbox,inner sep=2pt,minimum height=6mm,fill=white,minimum width=17mm]

\tikzstyle{medium dmap}=[draw,doubled,shape=NEbox,inner sep=2pt,minimum height=6mm,fill=white,minimum width=7mm]
\tikzstyle{medium dmap dag}=[draw,doubled,shape=SEbox,inner sep=2pt,minimum height=6mm,fill=white,minimum width=7mm]
\tikzstyle{medium dmap adj}=[draw,doubled,shape=SEbox,inner sep=2pt,minimum height=6mm,fill=white,minimum width=7mm]
\tikzstyle{medium dmap trans}=[draw,doubled,shape=SWbox,inner sep=2pt,minimum height=6mm,fill=white,minimum width=7mm]
\tikzstyle{medium dmap conj}=[draw,doubled,shape=NWbox,inner sep=2pt,minimum height=6mm,fill=white,minimum width=7mm]
\tikzstyle{semilarge dmap}=[draw,doubled,shape=NEbox,inner sep=2pt,minimum height=6mm,fill=white,minimum width=9.5mm]
\tikzstyle{semilarge dmap trans}=[draw,doubled,shape=SWbox,inner sep=2pt,minimum height=6mm,fill=white,minimum width=9.5mm]
\tikzstyle{semilarge dmap adj}=[draw,doubled,shape=SEbox,inner sep=2pt,minimum height=6mm,fill=white,minimum width=9.5mm]
\tikzstyle{semilarge dmap dag}=[draw,doubled,shape=SEbox,inner sep=2pt,minimum height=6mm,fill=white,minimum width=9.5mm]
\tikzstyle{semilarge dmap conj}=[draw,doubled,shape=NWbox,inner sep=2pt,minimum height=6mm,fill=white,minimum width=9.5mm]
\tikzstyle{large dmap}=[draw,doubled,shape=NEbox,inner sep=2pt,minimum height=6mm,fill=white,minimum width=12mm]
\tikzstyle{large dmap conj}=[draw,doubled,shape=NWbox,inner sep=2pt,minimum height=6mm,fill=white,minimum width=12mm]
\tikzstyle{large dmap trans}=[draw,doubled,shape=SWbox,inner sep=2pt,minimum height=6mm,fill=white,minimum width=12mm]
\tikzstyle{large dmap adj}=[draw,doubled,shape=SEbox,inner sep=2pt,minimum height=6mm,fill=white,minimum width=12mm]
\tikzstyle{large dmap dag}=[draw,doubled,shape=SEbox,inner sep=2pt,minimum height=6mm,fill=white,minimum width=12mm]
\tikzstyle{very large dmap}=[draw,doubled,shape=NEbox,inner sep=2pt,minimum height=6mm,fill=white,minimum width=19.5mm]

\tikzstyle{muxbox}=[draw,shape=rectangle,minimum height=3mm,minimum width=3mm,fill=white]
\tikzstyle{dmuxbox}=[muxbox,doubled]

\tikzstyle{box}=[draw,shape=rectangle,inner sep=2pt,minimum height=6mm,minimum width=6mm,fill=white]
\tikzstyle{dbox}=[draw,doubled,shape=rectangle,inner sep=2pt,minimum height=6mm,minimum width=6mm,fill=white]
\tikzstyle{dmap}=[draw,doubled,shape=NEbox,inner sep=2pt,minimum height=6mm,fill=white]
\tikzstyle{dmapdag}=[draw,doubled,shape=SEbox,inner sep=2pt,minimum height=6mm,fill=white]
\tikzstyle{dmapadj}=[draw,doubled,shape=SEbox,inner sep=2pt,minimum height=6mm,fill=white]
\tikzstyle{dmaptrans}=[draw,doubled,shape=SWbox,inner sep=2pt,minimum height=6mm,fill=white]
\tikzstyle{dmapconj}=[draw,doubled,shape=NWbox,inner sep=2pt,minimum height=6mm,fill=white]

\tikzstyle{ddmap}=[draw,doubled,dashed,shape=NEbox,inner sep=2pt,minimum height=6mm,fill=white]
\tikzstyle{ddmapdag}=[draw,doubled,dashed,shape=SEbox,inner sep=2pt,minimum height=6mm,fill=white]
\tikzstyle{ddmapadj}=[draw,doubled,dashed,shape=SEbox,inner sep=2pt,minimum height=6mm,fill=white]
\tikzstyle{ddmaptrans}=[draw,doubled,dashed,shape=SWbox,inner sep=2pt,minimum height=6mm,fill=white]
\tikzstyle{ddmapconj}=[draw,doubled,dashed,shape=NWbox,inner sep=2pt,minimum height=6mm,fill=white]

\boxshape{sNEbox}{0pt}{3pt}
\boxshape{sSEbox}{0pt}{-3pt}
\boxshape{sNWbox}{3pt}{0pt}
\boxshape{sSWbox}{-3pt}{0pt}
\tikzstyle{smap}=[draw,shape=sNEbox,fill=white]
\tikzstyle{smapdag}=[draw,shape=sSEbox,fill=white]
\tikzstyle{smapadj}=[draw,shape=sSEbox,fill=white]
\tikzstyle{smaptrans}=[draw,shape=sSWbox,fill=white]
\tikzstyle{smapconj}=[draw,shape=sNWbox,fill=white]

\tikzstyle{dsmap}=[draw,dashed,shape=sNEbox,fill=white]
\tikzstyle{dsmapdag}=[draw,dashed,shape=sSEbox,fill=white]
\tikzstyle{dsmaptrans}=[draw,dashed,shape=sSWbox,fill=white]
\tikzstyle{dsmapconj}=[draw,dashed,shape=sNWbox,fill=white]

\boxshape{mNEbox}{0pt}{10pt}
\boxshape{mSEbox}{0pt}{-10pt}
\boxshape{mNWbox}{10pt}{0pt}
\boxshape{mSWbox}{-10pt}{0pt}
\tikzstyle{mmap}=[draw,shape=mNEbox]
\tikzstyle{mmapdag}=[draw,shape=mSEbox]
\tikzstyle{mmaptrans}=[draw,shape=mSWbox]
\tikzstyle{mmapconj}=[draw,shape=mNWbox]

\tikzstyle{mmapgray}=[draw,fill=gray!40!white,shape=mNEbox]
\tikzstyle{smapgray}=[draw,fill=gray!40!white,shape=sNEbox]

\makeatletter

\pgfdeclareshape{cornerpoint}{
\inheritsavedanchors[from=rectangle] 
\inheritanchorborder[from=rectangle]
\inheritanchor[from=rectangle]{center}
\inheritanchor[from=rectangle]{north}
\inheritanchor[from=rectangle]{south}
\inheritanchor[from=rectangle]{west}
\inheritanchor[from=rectangle]{east}
\backgroundpath{
\southwest \pgf@xa=\pgf@x \pgf@ya=\pgf@y
\northeast \pgf@xb=\pgf@x \pgf@yb=\pgf@y

\pgfmathsetmacro{\pgf@shorten@left}{\pgfkeysvalueof{/tikz/shorten left}}
\pgfmathsetmacro{\pgf@shorten@right}{\pgfkeysvalueof{/tikz/shorten right}}

\pgfpathmoveto{\pgfpoint{0.5 * (\pgf@xa + \pgf@xb)}{\pgf@ya - 5pt}}
\pgfpathlineto{\pgfpoint{\pgf@xa - 8pt + \pgf@shorten@left}{\pgf@yb - 1.5 * \pgf@shorten@left}}
\pgfpathlineto{\pgfpoint{\pgf@xa - 8pt + \pgf@shorten@left}{\pgf@yb}}
\pgfpathlineto{\pgfpoint{\pgf@xb + 8pt - \pgf@shorten@right}{\pgf@yb}}
\pgfpathlineto{\pgfpoint{\pgf@xb + 8pt - \pgf@shorten@right}{\pgf@yb - 1.5 * \pgf@shorten@right}}
\pgfpathclose
}
}

\pgfdeclareshape{cornercopoint}{
\inheritsavedanchors[from=rectangle] 
\inheritanchorborder[from=rectangle]
\inheritanchor[from=rectangle]{center}
\inheritanchor[from=rectangle]{north}
\inheritanchor[from=rectangle]{south}
\inheritanchor[from=rectangle]{west}
\inheritanchor[from=rectangle]{east}
\backgroundpath{
\southwest \pgf@xa=\pgf@x \pgf@ya=\pgf@y
\northeast \pgf@xb=\pgf@x \pgf@yb=\pgf@y

\pgfmathsetmacro{\pgf@shorten@left}{\pgfkeysvalueof{/tikz/shorten left}}
\pgfmathsetmacro{\pgf@shorten@right}{\pgfkeysvalueof{/tikz/shorten right}}

\pgfpathmoveto{\pgfpoint{0.5 * (\pgf@xa + \pgf@xb)}{\pgf@yb + 5pt}}
\pgfpathlineto{\pgfpoint{\pgf@xa - 8pt + \pgf@shorten@left}{\pgf@ya + 1.5 * \pgf@shorten@left}}
\pgfpathlineto{\pgfpoint{\pgf@xa - 8pt + \pgf@shorten@left}{\pgf@ya}}
\pgfpathlineto{\pgfpoint{\pgf@xb + 8pt - \pgf@shorten@right}{\pgf@ya}}
\pgfpathlineto{\pgfpoint{\pgf@xb + 8pt - \pgf@shorten@right}{\pgf@ya + 1.5 * \pgf@shorten@right}}
\pgfpathclose
}
}

\makeatother

\pgfkeyssetvalue{/tikz/shorten left}{0pt}
\pgfkeyssetvalue{/tikz/shorten right}{0pt}

\tikzstyle{kpoint common}=[draw,fill=white,inner sep=1pt,minimum height=4mm]
\tikzstyle{kpoint sc}=[shape=cornerpoint,kpoint common]
\tikzstyle{kpoint adjoint sc}=[shape=cornercopoint,kpoint common]
\tikzstyle{kpoint}=[shape=cornerpoint,shorten left=5pt,kpoint common]
\tikzstyle{kpoint adjoint}=[shape=cornercopoint,shorten left=5pt,kpoint common]
\tikzstyle{kpoint conjugate}=[shape=cornerpoint,shorten right=5pt,kpoint common]
\tikzstyle{kpoint transpose}=[shape=cornercopoint,shorten right=5pt,kpoint common]
\tikzstyle{kpoint symm}=[shape=cornerpoint,shorten left=5pt,shorten right=5pt,kpoint common]

\tikzstyle{wide kpoint sc}=[shape=cornerpoint,kpoint common, minimum width=1 cm]
\tikzstyle{wide kpointdag sc}=[shape=cornercopoint,kpoint common, minimum width=1 cm]

\tikzstyle{black kpoint}=[shape=cornerpoint,shorten left=5pt,kpoint common,fill=black,font=\color{white}]

\tikzstyle{black kpoint sm}=[shape=cornerpoint,shorten left=5pt,kpoint common,fill=black,font=\color{white},scale=0.75]

\tikzstyle{black kpoint adjoint}=[shape=cornercopoint,shorten left=5pt,kpoint common,fill=black,font=\color{white}]
\tikzstyle{black kpointadj}=[shape=cornercopoint,shorten left=5pt,kpoint common,fill=black,font=\color{white}]

\tikzstyle{black kpointadj sm}=[shape=cornercopoint,shorten left=5pt,kpoint common,fill=black,font=\color{white},scale=0.75]

\tikzstyle{black dkpoint}=[shape=cornerpoint,shorten left=5pt,kpoint common,fill=black, doubled,font=\color{white}]
\tikzstyle{black dkpoint adjoint}=[shape=cornercopoint,shorten left=5pt,kpoint common,fill=black, doubled,font=\color{white}]
\tikzstyle{black dkpointadj}=[shape=cornercopoint,shorten left=5pt,kpoint common,fill=black, doubled,font=\color{white}]

\tikzstyle{black dkpoint sm}=[shape=cornerpoint,shorten left=5pt,kpoint common,fill=black, doubled,font=\color{white},scale=0.75]
\tikzstyle{black dkpointadj sm}=[shape=cornercopoint,shorten left=5pt,kpoint common,fill=black, doubled,font=\color{white},scale=0.75]

\tikzstyle{kpointdag}=[kpoint adjoint]
\tikzstyle{kpointadj}=[kpoint adjoint]
\tikzstyle{kpointconj}=[kpoint conjugate]
\tikzstyle{kpointtrans}=[kpoint transpose]

\tikzstyle{big kpoint}=[kpoint, minimum width=1.2 cm, minimum height=8mm, inner sep=4pt, text depth=3mm]

\tikzstyle{wide kpoint}=[kpoint, minimum width=1 cm, inner sep=2pt]
\tikzstyle{wide kpointdag}=[kpointdag, minimum width=1 cm, inner sep=2pt]
\tikzstyle{wide kpointconj}=[kpointconj, minimum width=1 cm, inner sep=2pt]
\tikzstyle{wide kpointtrans}=[kpointtrans, minimum width=1 cm, inner sep=2pt]

\tikzstyle{wider kpoint}=[kpoint, minimum width=1.25 cm, inner sep=2pt]
\tikzstyle{wider kpointdag}=[kpointdag, minimum width=1.25 cm, inner sep=2pt]
\tikzstyle{wider kpointconj}=[kpointconj, minimum width=1.25 cm, inner sep=2pt]
\tikzstyle{wider kpointtrans}=[kpointtrans, minimum width=1.25 cm, inner sep=2pt]

\tikzstyle{gray kpoint}=[kpoint,fill=gray!50!white]
\tikzstyle{gray kpointdag}=[kpointdag,fill=gray!50!white]
\tikzstyle{gray kpointadj}=[kpointadj,fill=gray!50!white]
\tikzstyle{gray kpointconj}=[kpointconj,fill=gray!50!white]
\tikzstyle{gray kpointtrans}=[kpointtrans,fill=gray!50!white]

\tikzstyle{gray dkpoint}=[kpoint,fill=gray!50!white,doubled]
\tikzstyle{gray dkpointdag}=[kpointdag,fill=gray!50!white,doubled]
\tikzstyle{gray dkpointadj}=[kpointadj,fill=gray!50!white,doubled]
\tikzstyle{gray dkpointconj}=[kpointconj,fill=gray!50!white,doubled]
\tikzstyle{gray dkpointtrans}=[kpointtrans,fill=gray!50!white,doubled]

\tikzstyle{white label}=[draw,fill=white,rectangle,inner sep=0.7 mm]
\tikzstyle{gray label}=[draw,fill=gray!50!white,rectangle,inner sep=0.7 mm]
\tikzstyle{black label}=[draw,fill=black,rectangle,inner sep=0.7 mm]

\tikzstyle{dkpoint}=[kpoint,doubled]
\tikzstyle{wide dkpoint}=[wide kpoint,doubled]
\tikzstyle{dkpointdag}=[kpoint adjoint,doubled]
\tikzstyle{wide dkpointdag}=[wide kpointdag,doubled]
\tikzstyle{dkcopoint}=[kpoint adjoint,doubled]
\tikzstyle{dkpointadj}=[kpoint adjoint,doubled]
\tikzstyle{dkpointconj}=[kpoint conjugate,doubled]
\tikzstyle{dkpointtrans}=[kpoint transpose,doubled]

\tikzstyle{kscalar}=[kpoint common, shape=EBox, inner xsep=-1pt, inner ysep=3pt,font=\small]
\tikzstyle{kscalarconj}=[kpoint common, shape=WBox, inner xsep=-1pt, inner ysep=3pt,font=\small]

\tikzstyle{spekpoint}=[kpoint sc,minimum height=5mm,inner sep=3pt]
\tikzstyle{spekcopoint}=[kpoint adjoint sc,minimum height=5mm,inner sep=3pt]

\tikzstyle{dspekpoint}=[spekpoint,doubled]
\tikzstyle{dspekcopoint}=[spekcopoint,doubled]

 \tikzstyle{upground}=[circuit ee IEC,thick,ground,rotate=90,scale=2.5]
 \tikzstyle{downground}=[circuit ee IEC,thick,ground,rotate=-90,scale=2.5]
 \tikzstyle{bigground}=[regular polygon,regular polygon sides=3,draw=gray,scale=0.50,inner sep=-0.5pt,minimum width=10mm,fill=gray]

\tikzstyle{arrs}=[-latex,font=\small,auto]
\tikzstyle{arrow plain}=[arrs]
\tikzstyle{arrow dashed}=[dashed,arrs]
\tikzstyle{arrow bold}=[very thick,arrs]
\tikzstyle{arrow hide}=[draw=white!0,-]
\tikzstyle{arrow reverse}=[latex-]
\tikzstyle{cdnode}=[]

\tikzstyle{discarding}=[fill=white, draw=black, shape=circle, style=upground]
\tikzstyle{smalldiscarding}=[fill=white, draw=black, style=upground, scale=0.5]
\tikzstyle{backdiscard}=[fill=white, draw=black, shape=circle, style=downground, scale=0.5]
\tikzstyle{smallbackdiscard}=[fill=white, draw=black, shape=circle, style=downground, scale=0.5]
\tikzstyle{state}=[fill=white, draw=black, style=triang, tikzit shape=rectangle]
\tikzstyle{kstate}=[fill=white, draw=black, style=kpoint, tikzit shape=rectangle]
\tikzstyle{kstateconj}=[fill=white, draw=black, style=kpoint conjugate, tikzit shape=rectangle]
\tikzstyle{kstateBIG}=[fill=white, draw=black, style=big kpoint, tikzit shape=rectangle]
\tikzstyle{effect}=[fill=white, draw=black, style=triangdag]
\tikzstyle{keffect}=[fill=white, draw=black, style=kpoint adjoint]
\tikzstyle{keffectconj}=[fill=white, draw=black, style=kpoint transpose]
\tikzstyle{morphdag}=[style=mapdag]
\tikzstyle{morph}=[style=hadamard]
\tikzstyle{WIDEmorph}=[style=hadamard, minimum width=14mm]
\tikzstyle{morphtrans}=[style=maptrans]
\tikzstyle{morphconj}=[style=mapconj]
\tikzstyle{CPMmorph}=[style=dmap]
\tikzstyle{CPMmorphconj}=[style=dmapconj]
\tikzstyle{CPMmorphdag}=[style=dmapdag]
\tikzstyle{CPMmorphtrans}=[style=dmaptrans]
\tikzstyle{CPMstate}=[fill=white, draw=black, style=triang, doubled]
\tikzstyle{CPMstateBIG}=[fill=white, draw=black, style={triang_lesssep}, doubled]
\tikzstyle{CPMkstate}=[fill=white, draw=black, style=kpoint, tikzit shape=rectangle, doubled]
\tikzstyle{CPMkstateconj}=[fill=white, draw=black, style=kpoint conjugate, tikzit shape=rectangle, doubled]
\tikzstyle{CPMkstateBIG}=[fill=white, draw=black, style=big kpoint, tikzit shape=rectangle, doubled]
\tikzstyle{CPMkeffect}=[fill=white, draw=black, style=kpoint adjoint, doubled]
\tikzstyle{CPMkeffectconj}=[fill=white, draw=black, style=kpoint transpose, doubled]
\tikzstyle{UHfB}=[fill=white, draw=black, style=triangdag, doubled, inner sep=-2pt]
\tikzstyle{leak}=[style=tinypoint, regular polygon rotate=-90]
\tikzstyle{leakfill}=[style=tinypoint, regular polygon rotate=-90, fill=black]
\tikzstyle{Z}=[style=dot, fill=green]
\tikzstyle{X}=[style=dot, fill=red]
\tikzstyle{black_dot}=[style=dot, fill=black]
\tikzstyle{white_dot}=[style=dot, fill=white]
\tikzstyle{qblack_dot}=[style=ddot, fill=black]
\tikzstyle{qwhite_dot}=[style=ddot, fill=white]
\tikzstyle{whitephase}=[style=wphase dot, fill=white]
\tikzstyle{qredphase}=[style=phase dot, fill=red]
\tikzstyle{qgreenphase}=[style=phase dot, fill=green]
\tikzstyle{had}=[style=hadamard, doubled]
\tikzstyle{box}=[style=hadamard]
\tikzstyle{classhad}=[style=hadamard]
\tikzstyle{antipode}=[style=anti]

\tikzstyle{dottededge}=[-,thick, densely dotted]
\tikzstyle{double edge}=[-, style=doubled, draw=black, tikzit draw={rgb,255: red,234; green,209; blue,255}]
\tikzstyle{new edge style 0}=[<-]
\tikzstyle{new edge style 1}=[-, draw={rgb,255: red,234; green,209; blue,255}, fill={rgb,255: red,234; green,209; blue,255}]
\tikzstyle{new edge style 1b}=[-, draw={rgb,255: red,208; green,218; blue,216}, fill={rgb,255: red,208; green,218; blue,216}]
\tikzstyle{new edge style 2}=[-, draw={rgb,255: red,14; green,188; blue,83}]
\tikzstyle{new edge style 3}=[<-, draw={rgb,255: red,234; green,209; blue,255}]
\tikzstyle{new edge style 4}=[<-, draw={rgb,255: red,0; green,128; blue,128}]
\tikzstyle{new edge style 5}=[-, draw={rgb,255: red,214; green,110; blue,62}]
\tikzstyle{new edge style 6}=[-, draw={rgb,255: red,174; green,20; blue,174}]

\usetikzlibrary{calc, positioning, shapes.geometric}
\usetikzlibrary{
	arrows,
	shapes,
	decorations,
	intersections,
	backgrounds,
	positioning,
	circuits.ee.IEC
	}

\newcommand{\tikzfigscale}[2]{\scalebox{#1}{\input{#2.tikz}}}

\newcommand{\cat}[1]{\mathcal{#1}}

\newcommand{\opcat}[1]{#1^\textrm{op}}

\newcommand{\seq}{\olessthan}

\makeatletter
\newcommand{\xRightarrow}[2][]{\ext@arrow 0359\Rightarrowfill@{#1}{#2}}
\makeatother

\usepackage{mathtools}
\usepackage{quiver}

\pgfsetlayers{background,strings,edgelayer,nodelayer,foreground,main}

\title{Monoidal su-categories}

\author{
Matt Wilson\\
\small Universit\'e Paris-Saclay, CNRS, ENS Paris-Saclay, Inria, CentraleSup\'elec,\\
\small Laboratoire M\'ethodes Formelles, France
\and
Giulio Chiribella\\
\small HKU-Oxford Joint Laboratory for Quantum Information and Computation
}
\date{}

\newtheorem{theorem}{Theorem}[section]

\theoremstyle{definition}
\newtheorem{definition}[theorem]{Definition}
\newtheorem{example}[theorem]{Example}
\theoremstyle{remark}

\begin{document}

\maketitle

\begin{abstract}
We introduce monoidal su-categories, an abstract categorical notion of single-input higher-order process over a monoidal category. The definition separates a base category $\mathcal{C}$ of lower-order processes from a monoidal category $\mathcal{V}$ of holes or supermaps and axiomatizes the compatibility needed for partial application to bipartite processes. For a fixed monoidal base $\mathcal{C}$, monoidal su-categories, monoidal su-functors, and monoidal su-natural transformations form a $2$-category $\mathbf{MonSuCat}_{\mathcal{C}}$. We then show that the category $\mathbf{Optic}[\mathcal{C}]$ of coend optics is $2$-initial in this $2$-category, giving a universal-property characterisation of coend optics as the minimal monoidal theory of single-hole contexts.
\end{abstract}


\section{Introduction}
Higher-order processes arise naturally across classical and quantum theoretical computer science. The framework of higher-order quantum operations provides a common language for quantum aspects of the agent-environment interaction in terms of quantum networks, quantum causal correlations, quantum processes with memory, and quantum games \cite{chiribella_supermaps,Chiribella2009TheoreticalNetworks,Bisio_2019,oreshkov,Pollock2015Non-MarkovianCharacterisation,Gutoski_2007}. In applied category theory, closely related hole constructions \cite{Riley2018CategoriesOptics,PickeringGW17,Clarke2020ProfunctorUpdate,Clarke2021ALenses,Lopez-Gonzalez2019TowardsLenses} support bidirectional transformations \cite{PickeringGW17,Foster2007CombinatorsProblem,Abou-Saleh2018IntroductionTransformations}, along with compositional approaches to classical game theory, learning, and cybernetics \cite{Ghani2016CompositionalTheory,Hedges2017CoherenceGames,FongST19,CapucciGHR22}. 

Many concrete constructions of holes are already available, including coend optics \cite{Riley2018CategoriesOptics,hefford_coend}, combs \cite{Chiribella2009TheoreticalNetworks,hefford_coend}, higher-order causal processes \cite{Kissinger2019AStructure,SimmonsKissinger2022,simmons_completelogic,jencova2024structurehigherorderquantum,Hoffreumon_2026}, and strong-natural transformations \cite{hefford_supermaps,wilson2026supermapsgeneralisedtheories,wilson_locality}. A unifying feature among these constructions is local application, more commonly referred to in the quantum-information literature as completeness. At the level of states, local application appears as complete positivity: an operation must remain valid when tensored with the identity on an arbitrary ancilla \cite{Nielsen2010QuantumInformation}. At higher order, the same principle underlies quantum superchannels \cite{chiribella_supermaps,Chiribella2009TheoreticalNetworks}, where a legitimate transformation must be applicable to part of any multipartite process, as in the following intuitive picture
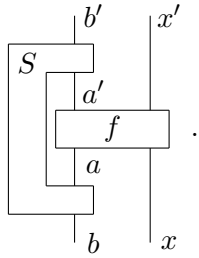
\[   \tikzfigscale{1}{mon_su_7b}. \]
Despite its salience, this completeness principle for higher-order transformations does not yet have a clear isolated categorical formulation.

To move towards a stable categorical algebra of holes, we isolate and formulate completeness, that is, local application, in categorical terms.
To this end, we introduce monoidal su-categories (where ``su'' refers to supermaps) as an abstract categorical model of single-input higher-order processes on a monoidal category. Our focus is on the categorical algebra needed to talk about local application, coherence, and comparison between different hole constructions.
\begin{theorem}
For every fixed monoidal category $\mathcal{C}$, monoidal su-categories over $\mathcal{C}$ form a $2$-category $\mathbf{MonSuCat}_{\mathcal{C}}$.
\end{theorem}

We also give a direct graphical interpretation of the axioms and use it to establish a structural theorem on categories of holes. More precisely, we prove the following.
\begin{theorem}
The monoidal su-category of coend optics $\mathbf{Optic}[\mathcal{C}]$ over a symmetric monoidal category $\mathcal{C}$ is a $2$-initial object in $\mathbf{MonSuCat}_{\mathcal{C}}$.
\end{theorem}

This work sits between two existing lines of approach. On one side, a number of higher-order constructions land in variants of $*$-autonomous and BV-like settings \cite{Kissinger2019AStructure,simmons_completelogic,SimmonsKissinger2022,hefford_supermaps,hefford2025bvcategoryspacetimeinterventions,wilson2026supermapsgeneralisedtheories,jencova2024structurehigherorderquantum,Hoffreumon_2026,Bisio_2019,COCKETT1997133,blute_BV, hirata_et_al:LIPIcs.LICS.2026.57}, where multi-input holes and local application can be studied simultaneously. On the other side, recent work attempts to isolate the \textit{essential} structure of such constructions for them to be interpreted as returning legitimate theories of higher-order processes, using enrichment \cite{wilson2023mathematical} together with polycategorical and Frobenius-like cotensors \cite{wilson2022freepolycategoriesunitarysupermaps,wilson2026higherordercircuits}, though at present this essential structure is treated only in the strict setting. Monoidal su-categories isolate the single-hole local-application fragment common to these settings. This further restriction allows us to tractably relax strictness, introducing coherences and defining functors and natural transformations for comparing and characterising theories of holes, a crucial step towards a full categorical algebra for holes which incorporates multiple inputs.



\section{Monoidal categories}
We recall briefly the usual abstract algebraic model for processes with both sequential and parallel composition; standard references are \cite{maclane1998categories,joyal_street1,coecke_kissinger_2017}. Throughout the paper, $\mathcal{C}$ denotes a base category of lower-order processes and $\mathcal{V}$ denotes a category of higher-order maps. We reserve boldface for named constructions such as $\mathbf{Optic}[\mathcal{C}]$, $\mathbf{Comb}[\mathcal{C}]$, $\mathbf{Slot}[\mathcal{C}]$, and $\mathbf{MonSuCat}_{\mathcal{C}}$. 

To specify a category $\mathcal{C}$ one gives a collection of objects $a,b,\dots$ together with a collection of morphisms $f:a \rightarrow b$ for each pair of objects $a,b$. The set of morphisms of type $a \rightarrow b$ is typically denoted $\mathcal{C}(a,b)$, and for each $a,b,c$ there is a sequential composition rule $\circ_{abc}: \mathcal{C}(a,b) \times \mathcal{C}(b,c) \rightarrow \mathcal{C}(a,c)$ which is associative in the sense that $(h \circ g) \circ f = h \circ (g \circ f)$. Note that subscripts on $\circ_{abc}$ are dropped whenever the context is clear. Sequential composition is also unital, meaning that there exists an identity morphism $1_a: a \rightarrow a$ for each object $a$ satisfying $1_b \circ f = f = f \circ 1_a$.

A strict monoidal category is a category equipped with a parallel composition on objects which lifts to morphisms. More precisely, each pair of objects $a,b$ defines a new object $a \otimes b$, and each pair of morphisms $f : a \rightarrow a'$ and $g : b \rightarrow b'$ defines a new morphism $f \otimes g : a \otimes b \rightarrow a' \otimes b'$. The tensor is associative so that $(a \otimes b) \otimes c = a \otimes (b \otimes c)$ and $(f \otimes g) \otimes h = f \otimes (g \otimes h)$, and it has a unit object $I$ such that $I \otimes a = a = a \otimes I$. Finally, sequential and monoidal composition interact through the interchange law
\[
(f \otimes g) \circ (f' \otimes g') = (f \circ f') \otimes (g \circ g').
\]

Most naturally occurring monoidal categories are not strict. In general monoidal categories associativity is relaxed up to a natural isomorphism $\alpha_{abc}: (a \otimes b) \otimes c \cong a \otimes (b \otimes c)$, and unitality is relaxed to the existence of natural isomorphisms $\lambda_a: I \otimes a \cong a$ and $\rho_a: a \otimes I \cong a$. These natural isomorphisms are required to satisfy coherence conditions referred to as the triangle and pentagon laws.

Monoidal categories find application in quantum information because they provide the abstract algebra behind circuit-based models of quantum computation and quantum processes \cite{coecke_kissinger_2017,Heunen2013IntroductionMechanics}. Circuit diagrams are sound and complete for monoidal categories \cite{joyal_street1,Selinger2009ACategories}, meaning that any equality derivable from purely topological moves on a circuit diagram is true by the defining equations of a monoidal category and vice versa. Circuit diagrams of this kind in category theory are typically referred to as string diagrams \cite{joyal_street1,Selinger2009ACategories}.

In string-diagrammatic notation, objects are drawn as wires and morphisms as boxes between those wires: \[ \tikzfigscale{1}{mon_cat_1}. \] Sequential and monoidal composition are denoted respectively by \[ \tikzfigscale{1}{mon_cat_2} \quad \text{and} \quad \tikzfigscale{1}{mon_cat_3}. \] Finally, the unit object $I$ is left unlabelled, being absorbed graphically as empty space.

\section{Monoidal su-categories}
We now introduce the single-hole structure studied in the rest of the paper. Intuitively, $\mathcal{C}$ contains ordinary processes, $\mathcal{V}$ contains higher-order maps between holes, the object $[a,a'] \in \mathcal{V}$ is the abstract hole with input $a$ and output $a'$, the isomorphism $\theta$ identifies ordinary processes with states of holes, and the structural map $\mu^{\boxtimes}$ expresses the possibility of applying a higher-order map locally to one tensor factor of a bipartite process.

\begin{definition}[Monoidal su-category]
A monoidal su-category consists of:
\begin{itemize}
\item A pair of monoidal categories $(\mathcal{V}, \mathcal{C})$
\item A functor $[-,=]: \mathcal{C}^{op} \times \mathcal{C} \rightarrow \mathcal{V}$
\item A natural isomorphism $\theta:\mathcal{C}(c,d) \cong \mathcal{V}(I_{\mathcal{V}}, [c,d])$ 
\item A natural isomorphism $\mu^{\boxtimes}_{aa'bb'}: [a \otimes b, a' \otimes b'] \rightarrow [a,a'] \boxtimes [b,b']$
\end{itemize}
such that $\theta(1_{I_{\mathcal{C}}}):I_{\mathcal{V}} \rightarrow [I_{\mathcal{C}},I_{\mathcal{C}}]$ is an isomorphism and the following diagrams commute:
\begin{itemize}
\item splitting property:  \[\begin{tikzcd}
		{I_{\mathcal{V}}} && {I_{\mathcal{V}} \boxtimes I_{\mathcal{V}}} \\
		{[a \otimes b, a \otimes b]} && {[a,a]\boxtimes[b,b]}
		\arrow["\lambda^{-1}_{I_{\mathcal{V}}}", from=1-1, to=1-3]
		\arrow["{\theta(1_a) \boxtimes \theta(1_b)}", from=1-3, to=2-3]
		\arrow["{\theta(1_{a \otimes b})}"', from=1-1, to=2-1]
		\arrow["{\mu^{\boxtimes}_{(a,b), (a,b)}}"', from=2-1, to=2-3]
\end{tikzcd}\]
\item mutual coherence between lower and higher associators \[
\begin{adjustbox}{max width=\textwidth}
\begin{tikzcd}
	& {[c \otimes(d \otimes e), (c' \otimes d') \otimes e']} \\
	{[(c \otimes d) \otimes e, (c' \otimes d') \otimes e']} && {[c \otimes(d \otimes e), c' \otimes (d' \otimes e')]} \\
	{[c \otimes d,c' \otimes d'] \boxtimes [ e,e']} && {[c,c'] \boxtimes [d \otimes e,d' \otimes e']} \\
	{([c,c'] \boxtimes [d,d']) \boxtimes [e,e']} && {[c,c'] \boxtimes ([d,d'] \boxtimes [e,e'])}
	\arrow["{[\alpha_{cde}, 1]}"', from=1-2, to=2-1]
	\arrow["{\mu_{(c \otimes d)( c' \otimes d') ee'}}"', from=2-1, to=3-1]
	\arrow["{\mu_{cc'dd'} \boxtimes 1}"', from=3-1, to=4-1]
	\arrow["{[1, \alpha_{c'd'e'}]}"', from=1-2, to=2-3]
	\arrow["{\mu_{cc'  (d\otimes e) (d' \otimes e')}}"', from=2-3, to=3-3]
	\arrow["{1 \boxtimes \mu_{dd'ee'}}"', from=3-3, to=4-3]
	\arrow["{\alpha_{[c,c'][d,d'][e,e']}}"', from=4-1, to=4-3]
\end{tikzcd}
\end{adjustbox}
\]
\item mutual coherence between lower and higher unitors: \[\begin{tikzcd}
		{[I_{\mathcal{C}} \otimes c, I_{\mathcal{C}} \otimes d]} && {[c,I_{\mathcal{C}} \otimes d]} \\
		{[I_{\mathcal{C}}, I_{\mathcal{C}}] \boxtimes [c,d]} && {[c,d]} \\
		& {I_{\mathcal{V}} \boxtimes [c,d]}
		\arrow["{[\lambda_c,1_{I_{\mathcal{C}} \otimes d}]}"', from=1-3, to=1-1]
		\arrow["{[1_c,\lambda_d]}", from=1-3, to=2-3]
		\arrow["{\mu^{\boxtimes}_{(I_{\mathcal{C}},c),(I_{\mathcal{C}},d)}}"', from=1-1, to=2-1]
		\arrow["{\theta(1_{I_{\mathcal{C}}})^{-1} \boxtimes 1}"', from=2-1, to=3-2]
		\arrow["{\lambda_{[c,d]}}"', from=3-2, to=2-3]
\end{tikzcd}\]
\end{itemize}
\end{definition}


Let us now unpack and directly interpret each feature of the definition of monoidal-su categories. 
In line with the enrichment-based approach of \cite{wilson2023mathematical}, which models multi-hole structure, a monoidal su-category separates the lower-order process theory $\mathcal{C}$ from the category $\mathcal{V}$ of holes or supermaps. The functor $[-,=]$ sends each pair of objects $a,a'$ in $\mathcal{C}$ to the hole object $[a,a']$ in $\mathcal{V}$:
\[   \tikzfigscale{1}{hole_pic_0} \quad \sim \quad  \ \tikzfigscale{1}{hole_pic_1}.  \]

Morphisms of type $S: [a,a'] \rightarrow [b,b']$ then represent maps between holes, or supermaps, and through the isomorphism $\theta$ every state of type $I_{\mathcal{V}} \rightarrow [a,a']$ corresponds to an ordinary morphism $a \rightarrow a'$ in $\mathcal{C}$:
\[   \tikzfigscale{1}{mon_su_1} \quad \sim \quad  \ \tikzfigscale{1}{mon_su_2}.\quad \quad \quad \quad  \tikzfigscale{1}{mon_su_3} \quad \sim \quad  \ \tikzfigscale{1}{mon_su_4} \quad \sim \quad  \ \tikzfigscale{1}{mon_su_5}.  \]
Note that composition in $\mathcal{V}$ represents nesting of supermaps between holes \[     \tikzfigscale{1}{mon_su_nest_0} \quad \sim \quad  \ \tikzfigscale{1}{mon_su_nest_1}.  \]
The isomorphism $\mu^{\boxtimes}$ is the key feature for modelling completeness: it means that states of type $I_{\mathcal{V}} \rightarrow [a,a'] \boxtimes [b,b']$ precisely correspond to bipartite morphisms in $\mathcal{C}$, that is, to morphisms of type $a \otimes b \rightarrow a' \otimes b'$. This gives categorical semantics to partial application on arbitrary bipartite morphisms:
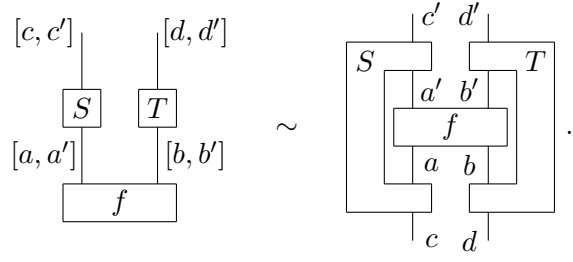
\[   \tikzfigscale{1}{mon_su_6} \quad \sim \quad  \ \tikzfigscale{1}{mon_su_7}. \]

The remaining conditions in the definition of a monoidal su-category are coherence conditions, so they do not admit much direct diagrammatic interpretation beyond identifying different ways of reading empty space and formal bracketing.
\begin{example}[Optics]\label{ex:optics}
Following \cite{Riley2018CategoriesOptics,hefford_coend}, the symmetric monoidal category $\mathbf{Optic}[\mathcal{C}]$ has object pairs $(a,a')$ of objects of $\mathcal{C}$ and hom-sets \[\mathbf{Optic}[\mathcal{C}]((a,a'),(b,b')):=\int^{m \in \mathcal{C}}\mathcal{C}(b,m \otimes a)\times \mathcal{C}(m \otimes a',b').\]
Concretely, a morphism is represented by an auxiliary object $m$ together with a pair of maps
\[( f:b \longrightarrow m \otimes a, g:m \otimes a' \longrightarrow b' ), \]
quotiented by equivalence under sliding, that is, equivalence under $( (1 \otimes u) f, g) \cong ( f, g(1 \otimes u)) $. Graphically, coend optics can be denoted by
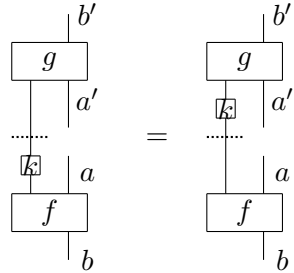
\[\tikzfigscale{1}{intuit_optic_1} \quad =  \quad \tikzfigscale{1}{intuit_optic_2}\]
With parallel composition on objects defined by $(c,c') \otimes (d,d') := (c \otimes d, c' \otimes d')$ and with sequential and parallel composition of morphisms defined by \[\tikzfigscale{1}{optic_comp_2} \quad \circ \quad  \tikzfigscale{1}{optic_comp_1} \quad = \quad \tikzfigscale{1}{optic_comp_3}, \quad \quad \tikzfigscale{1}{optic_comp_4} \quad \otimes \quad  \tikzfigscale{1}{optic_comp_1}  \quad = \quad  \tikzfigscale{1}{optic_comp_5}   \]
The symmetric monoidal category $\mathbf{Optic}[\mathcal{C}]$ can be equipped with the structure of a symmetric monoidal su-category as follows.
\begin{itemize}
    \item The functor $[-,=]_{\cong}:\mathcal{C}^{op} \times \mathcal{C} \rightarrow \mathbf{Optic}[\mathcal{C}]$ is the identity on objects and sends a pair of morphisms $(f:b \rightarrow a,g:a' \rightarrow b')$ to the corresponding memoryless optic, with auxiliary object $I_{\mathcal{C}}$, induced by $f$ and $g$ through the unitors of $\mathcal{C}$.
    \item The natural isomorphism $\mathbf{Optic}[\mathcal{C}]((I_{\mathcal{C}},I_{\mathcal{C}}),(c,c')) \cong \mathcal{C}(c,c')$, where a state is represented by an auxiliary-free optic and the unitors identify that representative with an ordinary morphism $c \rightarrow c'$.
    \item Noting that $[a,a'] \boxtimes [b,b'] = (a,a') \boxtimes (b,b') = (a \otimes b , a' \otimes b') = [a \otimes b , a' \otimes b']$, we take $\mu_{aa'bb'}$ to be the identity morphism.
\end{itemize}
Graphically, these assignments are represented by \[ [f,g] \quad \mapsto \quad (f,g) := \tikzfigscale{1}{optic_assign_1}, \qquad f \quad \mapsto \quad \tikzfigscale{1}{optic_assign_2} = \tikzfigscale{1}{optic_assign_3}. \]
The required coherences are inherited from the coherence isomorphisms in $\mathcal{C}$, and the splitting law reduces to \[ \tikzfigscale{1}{optic_assign_4} \boxtimes \tikzfigscale{1}{optic_assign_5} = \tikzfigscale{1}{optic_assign_6}. \]
\end{example}

\begin{example}[Quantum superchannels]
Traditionally, the definition of quantum superchannels is given directly in terms of the Choi--Jamio{\l}kowski isomorphism \cite{chiribella_supermaps}. To keep the presentation self-contained, we instead follow \cite{wilson2023compositional} and use simple isomorphisms available in finite-dimensional Hilbert spaces, which are in general used to construct the Choi--Jamio{\l}kowski isomorphism. We note that the set of quantum channels $\mathbf{QC}(A,A')$ from Hilbert space $A$ to Hilbert space $A'$ is a subset $\mathbf{QC}(A,A') \subseteq \mathbf{L}(\mathbf{L}(A),\mathbf{L}(A'))$ of the set of linear maps from $\mathbf{L}(A)$ to $\mathbf{L}(A')$, where for any $X$, $\mathbf{L}(X) := \mathbf{L}(X,X)$. The local application $\bigl(S \otimes_u i\bigr)(\mathcal{E})$ of any linear map $S \colon \mathbf{L}(\mathbf{L}(A),\mathbf{L}(A')) \to \mathbf{L}(\mathbf{L}(B),\mathbf{L}(B'))$ to an element
\[
\mathcal{E} \in \mathbf{QC}(A \otimes X, A' \otimes X')
\subseteq \mathbf{L}(\mathbf{L}(A \otimes X), \mathbf{L}(A' \otimes X'))
\]
can be defined using the following isomorphisms (which are inherited from the compact closure of the category $\mathbf{L}$ of finite-dimensional Hilbert spaces, giving for each $A,B$ a natural isomorphism $u_{AB} : \mathbf{L}(A \otimes B , A \otimes B) \cong \mathbf{L}(A  , A ) \otimes_{\mathbf{L}} \mathbf{L}(B ,  B)$):
\begin{equation}
\begin{tikzcd}[column sep=large, row sep=large]
\mathbf{L}(\mathbf{L}(A),\mathbf{L}(A')) \otimes \mathbf{L}(\mathbf{L}(X),\mathbf{L}(X'))
  \arrow[r, "S \otimes i"]
&
\mathbf{L}(\mathbf{L}(B),\mathbf{L}(B')) \otimes \mathbf{L}(\mathbf{L}(X),\mathbf{L}(X'))
\\
\mathbf{L}(\mathbf{L}(A)\otimes \mathbf{L}(X), \mathbf{L}(A')\otimes \mathbf{L}(X'))
  \arrow[u, "u^{-1}_{LALA'LXLX'}"']
&
\mathbf{L}(\mathbf{L}(B)\otimes \mathbf{L}(X), \mathbf{L}(B')\otimes \mathbf{L}(X'))
  \arrow[u, "u_{LBLB'LXLX'}"]
\\
\mathbf{L}(\mathbf{L}(A\otimes X), \mathbf{L}(A' \otimes X'))
  \arrow[u, "{\mathbf{L}(u_{AX},u^{-1}_{A'X'})}"']
  \arrow[r, "(S \otimes_u i)"']
&
\mathbf{L}(\mathbf{L}(B\otimes X), \mathbf{L}(B' \otimes X'))
  \arrow[u, "{\mathbf{L}(u^{-1}_{BX},u_{B'X'})}"]
\end{tikzcd},
\tag{2.8}
\end{equation}
A linear map $S$ is a quantum superchannel if, for every $\mathcal{E} \in \mathbf{QC}(A \otimes X, A' \otimes X')$, one has
\[
\bigl(S \otimes_u i\bigr)(\mathcal{E}) \in \mathbf{QC}(B \otimes X, B' \otimes X').
\]
In other words, a quantum superchannel is a linear map that \emph{completely} preserves quantum channels: it can be \emph{locally applied} to a channel and is guaranteed to return a new quantum channel.
\end{example}
In the appendix we give several further constructions of theories of higher-order maps that form monoidal su-categories.




\subsection{The 2-category of monoidal su-categories}
We now define the morphisms and $2$-morphisms that compare monoidal su-categories over a fixed lower-order theory $\mathcal{C}$ while allowing the higher-order category $\mathcal{V}$ to vary.

\begin{theorem}[The $2$-category $\mathbf{MonSuCat}_{\mathcal{C}}$]\label{thm:monsucat-2cat}
For every monoidal category $\mathcal{C}$, the following data define a $2$-category $\mathbf{MonSuCat}_{\mathcal{C}}$.
\begin{itemize}
\item $0$-cells are monoidal su-categories $(\mathcal{V},\mathcal{C})$.
\item $1$-cells are \emph{monoidal su-functors} $(\mathcal{F},\mu^{\mathcal{F}}):(\mathcal{V},\mathcal{C}) \rightarrow (\mathcal{V}',\mathcal{C})$, where $\mathcal{F}:\mathcal{V}\rightarrow\mathcal{V}'$ is a monoidal functor and $\mu^{\mathcal{F}}_{a,a'}:[a,a'] \rightarrow \mathcal{F}[a,a']$ is a natural isomorphism such that, writing $\zeta_{cc'dd'} := (\mu^{\boxtimes}_{cc'dd'})^{-1}$ and $(u^{\mathcal{F}},\phi^{\mathcal{F}})$ for the monoidal structure of $\mathcal{F}$, the diagrams
\[
\begin{tikzcd}
I_{\mathcal{V}'} && {[I_{\mathcal{C}},I_{\mathcal{C}}]} \\
\mathcal{F}(I_{\mathcal{V}}) && \mathcal{F}[I_{\mathcal{C}},I_{\mathcal{C}}]
\arrow["{\theta'(1_{I_{\mathcal{C}}})}", from=1-1, to=1-3]
\arrow["{u^{\mathcal{F}}}"', from=1-1, to=2-1]
\arrow["{\mu^{\mathcal{F}}_{I_{\mathcal{C}},I_{\mathcal{C}}}}", from=1-3, to=2-3]
\arrow["{\mathcal{F}(\theta(1_{I_{\mathcal{C}}}))}"', from=2-1, to=2-3]
\end{tikzcd}
\]
and
\[
\begin{tikzcd}[column sep=large]
{[c \otimes d,c' \otimes d']} &&& {[c,c'] \boxtimes [d,d']} \\
{\mathcal{F}[c \otimes d,c' \otimes d']} &&& {\mathcal{F}[c,c'] \boxtimes \mathcal{F}[d,d']}
\arrow["{\mu^{\mathcal{F}}_{c \otimes d,c' \otimes d'}}"', from=1-1, to=2-1]
\arrow["{\zeta_{cc'dd'}}", from=1-4, to=1-1]
\arrow["{\mu^{\mathcal{F}}_{c,c'} \boxtimes \mu^{\mathcal{F}}_{d,d'}}", from=1-4, to=2-4]
\arrow["{\mathcal{F}(\zeta_{cc'dd'}) \circ \phi^{\mathcal{F}}_{[c,c'],[d,d']}}", from=2-4, to=2-1]
\end{tikzcd}
\]
commute. Composition is given by
\[\mu^{\mathcal{G}\mathcal{F}} := \mathcal{G}(\mu^{\mathcal{F}})\circ \mu^{\mathcal{G}}.\]
and 
\[u^{\mathcal{G}\mathcal{F}} := \mathcal{G}(u^{\mathcal{F}})\circ u^{\mathcal{G}}.\]
\item $2$-cells are \emph{monoidal su-natural transformations} $\eta:(\mathcal{F},\mu^{\mathcal{F}})\Rightarrow(\mathcal{G},\mu^{\mathcal{G}})$, that is, monoidal natural transformations $\eta:\mathcal{F}\Rightarrow\mathcal{G}$ satisfying
\[
\begin{tikzcd}
& {[c,c']} & \\
{\mathcal{F}[c,c']} && {\mathcal{G}[c,c']}
\arrow["{\mu^{\mathcal{F}}_{c,c'}}"', from=1-2, to=2-1]
\arrow["{\mu^{\mathcal{G}}_{c,c'}}", from=1-2, to=2-3]
\arrow["{\eta_{[c,c']}}"', from=2-1, to=2-3]
\end{tikzcd}
\]
for every $c,c' \in \mathcal{C}$.
\end{itemize}
\end{theorem}
\begin{proof}
Given in Appendix~\ref{app:monsucat-2cat}.
\end{proof}

Morphisms in the fixed-base $2$-category $\mathbf{MonSuCat}_{\mathcal{C}}$ give a direct way to compare algebraic theories of holes over the same lower-order process theory $\mathcal{C}$. As we will see, they give a way to frame the idea that coend-optics can be embedded inside \textit{any} theory of concurrent holes. 
\section{Graphical notation for monoidal su-categories}
In this section we introduce a basic string-diagrammatic notation for monoidal su-categories and monoidal su-functors. The notation is intentionally conservative: a supermap of type $[a,a'] \rightarrow [b,b']$ is drawn as an ordinary process in $\mathcal{V}$ rather than as a literal box with a hole. 
A sound and complete hole-based graphical calculus is left for future work.

Recall that a monoidal su-category consists of a pair of monoidal categories $(\mathcal{V}, \mathcal{C})$. We draw string diagrams in $\mathcal{V}$, but use the functor $[-,=]: \mathcal{C}^{op} \times \mathcal{C} \rightarrow \mathcal{V}$ to label certain wires by hole objects. More precisely, we depict the action of $[-,=]$ on a pair of morphisms $(f,g)$ by
\[
[f,g] := \quad \tikzfigscale{1}{diagram_functor}.
\]
We write the natural isomorphism $\theta:\mathcal{C}(c,d) \cong \mathcal{V}(I_{\mathcal{V}}, [c,d])$ as
\[
\theta(f) := \quad \tikzfigscale{1}{diagram_state} \quad = \quad \tikzfigscale{1}{diagram_state_2}.
\]
We write the structural isomorphism $\mu^{\boxtimes}_{aa'bb'}: [a \otimes b, a' \otimes b'] \rightarrow [a,a'] \boxtimes [b,b']$ as
\[ \mu^{-1} := \quad \tikzfigscale{1}{diagram_mu_1} \qquad \mu := \quad \tikzfigscale{1}{diagram_mu_2}, \]
with inverse equations
\[ \tikzfigscale{1}{diagram_mu_3} \quad = \quad \tikzfigscale{1}{diagram_mu_4}, \qquad \tikzfigscale{1}{diagram_mu_5} \quad = \quad \tikzfigscale{1}{diagram_mu_6}, \]
and naturality displayed by
\[ \tikzfigscale{1}{diagram_mu_natural_1} \quad = \quad \tikzfigscale{1}{diagram_mu_natural_2}. \]
The unit state $\theta(1_{I_{\mathcal{C}}}):I_{\mathcal{V}} \rightarrow [I_{\mathcal{C}},I_{\mathcal{C}}]$ is also an isomorphism; we depict its inverse as
\[ \theta(1_{I_{\mathcal{C}}})^{-1} := \quad \tikzfigscale{1}{diagram_state_flip} \]
so that
\[ \tikzfigscale{1}{diagram_state_slip_eq} \quad = \quad 1. \]
The axioms used repeatedly later are the splitting law
\[ \tikzfigscale{1}{diagram_mu_id_1} \quad = \quad \tikzfigscale{1}{diagram_mu_id_2}, \]
associativity for splitting
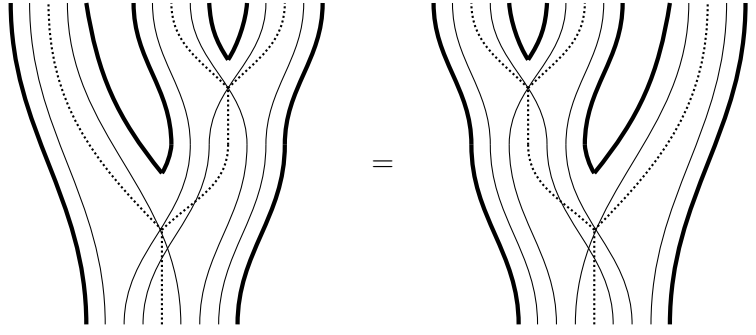
\[ \tikzfigscale{1}{diagram_mu_ass_1} \quad = \quad \tikzfigscale{1}{diagram_mu_ass_2}, \]
and unitality for splitting
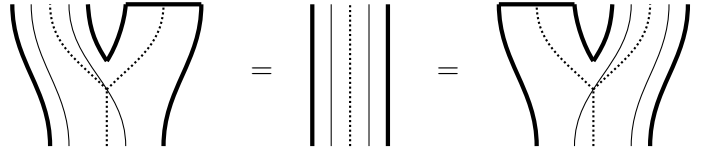
\[ \tikzfigscale{1}{diagram_mu_unit_1} \quad = \quad \tikzfigscale{1}{diagram_mu_unit_2} \quad = \quad \tikzfigscale{1}{diagram_mu_unit_3}. \]

A monoidal su-functor can then be expressed graphically using the following functor-box notation, subject to the additional properties
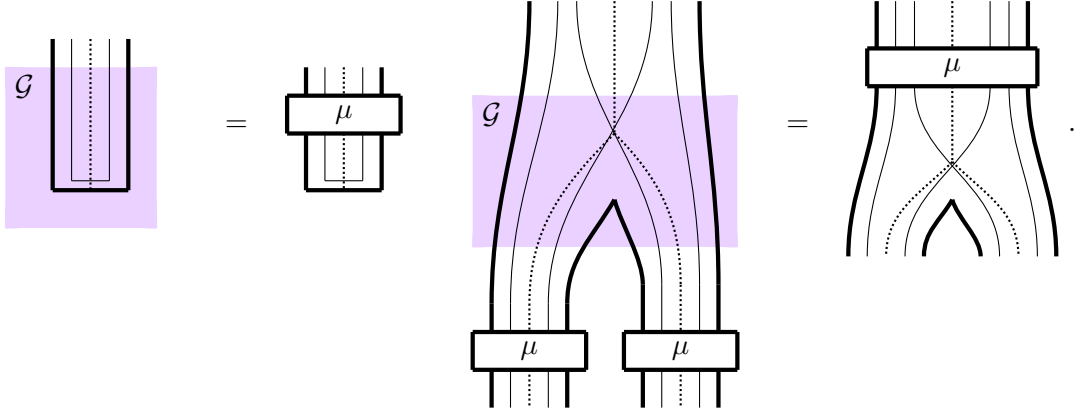
\[ \tikzfigscale{1}{functor_box_1} \quad = \quad \tikzfigscale{1}{functor_box_2} \qquad \tikzfigscale{1}{functor_box_3} \quad = \quad \tikzfigscale{1}{functor_box_4}. \]

\section{A universal property for coend optics}
In what follows we fix a symmetric monoidal category $\mathcal{C}$ and work in $\mathbf{MonSuCat}_{\mathcal{C}}$, so every $1$-cell acts as the identity on the lower-order category $\mathcal{C}$. We therefore write an object simply as $\mathcal{V}$ when no confusion can arise.

We can now state the second main theorem of the paper, which identifies coend optics as the universal minimal theory of single-hole contexts over $\mathcal{C}$.
\begin{theorem}[Initiality of optics]\label{thm:optic-initiality}
$\mathbf{Optic}[\mathcal{C}]$ is a $2$-initial object in $\mathbf{MonSuCat}_{\mathcal{C}}$.
\end{theorem}
\begin{proof}
Define a functor
\[\mathcal{O}_{\mathcal{V}}:\mathbf{Optic}[\mathcal{C}] \longrightarrow \mathcal{V}\]
on objects by \[\mathcal{O}_{\mathcal{V}}(a,a') := [a,a'].\] For a representative $(f,g,m)$ of an optic $(a,a') \rightarrow (b,b')$, with \[ f:b \rightarrow m \otimes a, \qquad g:m \otimes a' \rightarrow b', \] define the corresponding morphism in $\mathcal{V}$ by the graphical composite \[ \mathcal{O}_{m}(f,g) := \quad \tikzfigscale{1}{optic_def}. \]

To see that this assignment descends to the coend, first verify the corresponding commutative diagram for the cowedge:
\[\begin{tikzcd}
	{\mathcal{C}(b,z \otimes a) \times \mathcal{C}(z' \otimes a', b')} && {\mathcal{C}(b,z' \otimes a) \times \mathcal{C}(z' \otimes a', b')} \\
	{\mathcal{C}(b,z \otimes a) \times \mathcal{C}(z \otimes a', b')} && {\mathcal{V}([a , a'],[b, b'])}
	\arrow["{\mathcal{C}(b,u \otimes 1_a) \times 1}", from=1-1, to=1-3]
	\arrow["{1 \times \mathcal{C}(u \otimes 1_{a'}, b')}"', from=1-1, to=2-1]
	\arrow["{\mathcal{O}_{(a,a'),(b,b'),z'}}", from=1-3, to=2-3]
	\arrow["{\mathcal{O}_{(a,a'),(b,b'),z}}"', from=2-1, to=2-3]
\end{tikzcd}\]
and hence obtain the unique morphism
\[
\int^{z \in \mathcal{C}}\mathcal{C}(b,z \otimes a) \times \mathcal{C}(z \otimes a', b')
\longrightarrow
\mathcal{V}([a , a'],[b, b']).
\]
In string-diagrammatic language, this is precisely the sliding equation
\[
\tikzfigscale{1}{string_optic_1}
\quad = \quad
\tikzfigscale{1}{string_optic_2}.
\]

This defines $\mathcal{O}_{\mathcal{V}}$ on hom-sets. We now check that it is functorial. For preservation of the identity, the relevant commutative diagram reads:
\[\begin{tikzcd}
	{[c,c']} && {[c,c']} \\
	& {I_{\mathcal{V}} \boxtimes[c,c']} \\
	& {[I_{\mathcal{C}},I_{\mathcal{C}}] \boxtimes[c,c']} \\
	& {[I_{\mathcal{C}} \otimes c,I_{\mathcal{C}} \otimes c']} \\
	{[c,c']} && {[c,c']}
	\arrow["{1_{[c,c']}}", from=1-1, to=1-3]
	\arrow["{\lambda_{[c,c']}}"{description}, from=1-1, to=2-2]
	\arrow["{\mathcal{O}_{\mathcal{V}}(1_{(c,c')})}"', curve={height=30pt}, from=1-1, to=5-1]
	\arrow["{\lambda_{[c,c']}}"{description}, from=1-3, to=2-2]
	\arrow["{1_{[c,c']}}", curve={height=-30pt}, from=1-3, to=5-3]
	\arrow["{\theta(1_{I_{\mathcal{C}}}) \boxtimes 1_{[c,c']}}"', from=2-2, to=3-2]
	\arrow["{\mu^{\boxtimes}_{(I_{\mathcal{C}},c),(I_{\mathcal{C}},c')}}"', from=3-2, to=4-2]
	\arrow["{[\lambda_c, \lambda_{c'}]}"{description}, from=4-2, to=5-1]
	\arrow["{[\lambda_c, \lambda_{c'}]}"{description}, from=4-2, to=5-3]
	\arrow["{1_{[c,c']}}"', from=5-1, to=5-3]
\end{tikzcd}\]
Diagrammatically, this simply reads
\[\tikzfigscale{1}{diagram_mu_unit_1_flip} \quad = \quad \tikzfigscale{1}{diagram_mu_unit_2}. \]

For preservation of composition, let $\omega:(c,c') \rightarrow (d,d')$ be represented by $(f,g,z)$ and let $\omega':(d,d') \rightarrow (e,e')$ be represented by $(f',g',y)$. Their composite coend optic has auxiliary object $y \otimes z$ given concretely by
\[ e \xrightarrow{f'} y \otimes d \xrightarrow{1_y \otimes f} y \otimes (z \otimes c) \xrightarrow{\alpha^{-1}_{yzc}} (y \otimes z) \otimes c \]
and
\[ (y \otimes z) \otimes c' \xrightarrow{\alpha_{yzc'}} y \otimes (z \otimes c') \xrightarrow{1_y \otimes g} y \otimes d' \xrightarrow{g'} e'. \]
The commutative-diagram verification of \[ \mathcal{O}_{\mathcal{V}}(\omega' \circ \omega) = \mathcal{O}_{\mathcal{V}}(\omega') \circ \mathcal{O}_{\mathcal{V}}(\omega) \] is:
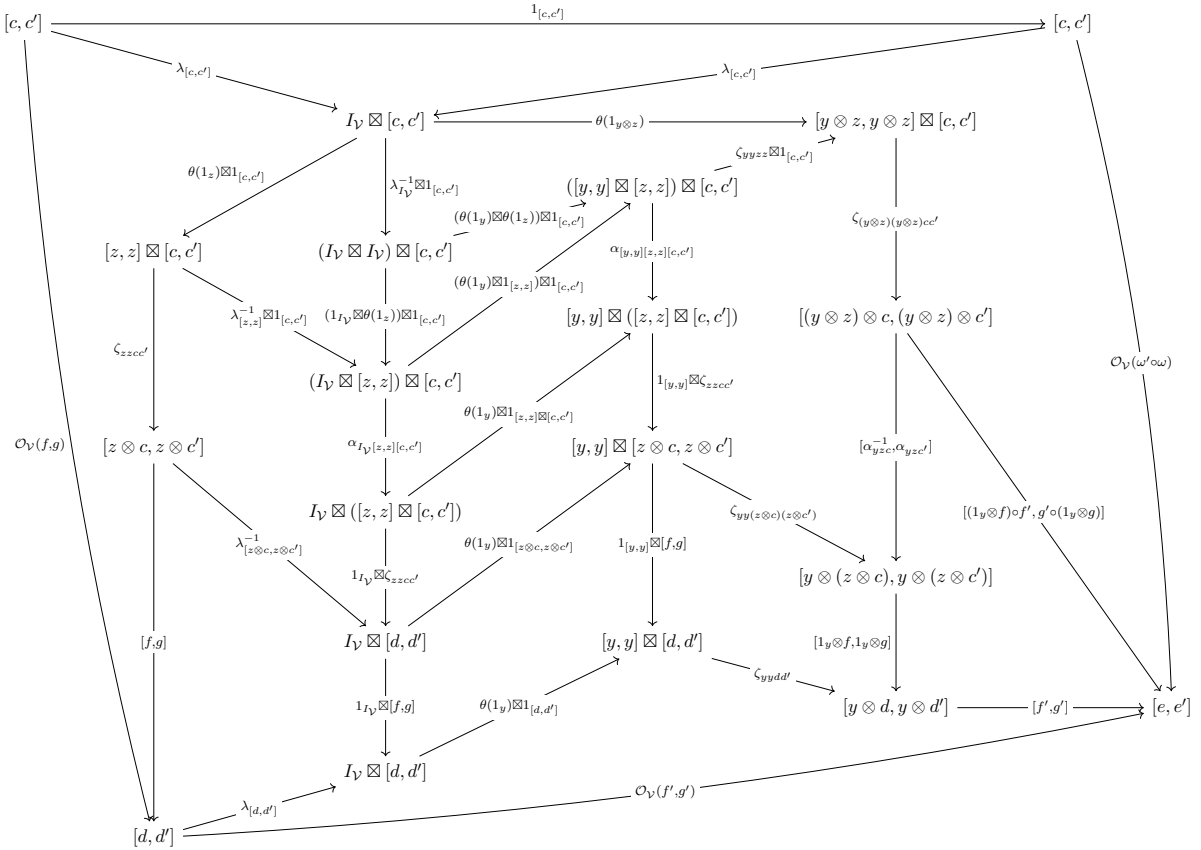
\[
\begin{adjustbox}{max width=\textwidth}
\begin{tikzcd}
	{[c,c']} &&&&&&& {[c,c']} & \\
	\\
	&&& {I_{\mathcal{V}} \boxtimes[c,c']} &&& {[y \otimes z,y \otimes z] \boxtimes[c,c']} \\
	&&&&& {([y,y] \boxtimes[z,z]) \boxtimes[c,c']} \\
	& {[z,z] \boxtimes[c,c']} && {(I_{\mathcal{V}} \boxtimes I_{\mathcal{V}}) \boxtimes[c,c']} \\
	&&&&& {[y,y]\boxtimes ([z,z] \boxtimes[c,c'])} & {[(y \otimes z) \otimes c,(y \otimes z) \otimes c']} \\
	&&& {(I_{\mathcal{V}} \boxtimes[z,z]) \boxtimes[c,c']} \\
	& {[z \otimes c,z \otimes c']} &&&& {[y,y]\boxtimes [z \otimes c,z \otimes c']} \\
	&&& {I_{\mathcal{V}} \boxtimes ([z,z] \boxtimes[c,c'])} \\
	&&&&&& {[y \otimes (z \otimes c),y \otimes (z \otimes c')]} \\
	&&& {I_{\mathcal{V}} \boxtimes[d,d']} && {[y,y] \boxtimes[d,d']} \\
	&&&&&& {[y \otimes d,y \otimes d']} && {[e,e']} \\
	&&& {I_{\mathcal{V}} \boxtimes[d,d']} \\
	& {[d,d']}
	\arrow["{1_{[c,c']}}", from=1-1, to=1-8]
	\arrow["{\lambda_{[c,c']}}"{description}, from=1-1, to=3-4]
	\arrow["{\mathcal{O}_{\mathcal{V}}(f,g)}"', curve={height=30pt}, from=1-1, to=14-2]
	\arrow["{\lambda_{[c,c']}}"{description}, from=1-8, to=3-4]
	\arrow["{\mathcal{O}_{\mathcal{V}}(\omega' \circ \omega)}"{description}, curve={height=-30pt}, from=1-8, to=12-9]
	\arrow["{\theta(1_{y \otimes z})}"{description}, from=3-4, to=3-7]
	\arrow["{\theta(1_{z}) \boxtimes 1_{[c,c']}}"', from=3-4, to=5-2]
	\arrow["{\lambda^{-1}_{I_{\mathcal{V}}} \boxtimes 1_{[c,c']}}", from=3-4, to=5-4]
	\arrow["{\zeta_{(y \otimes z)(y \otimes z)cc'}}"{description}, from=3-7, to=6-7]
	\arrow["{\zeta_{yyzz} \boxtimes 1_{[c,c']}}"{description}, from=4-6, to=3-7]
	\arrow["{\alpha_{[y,y][z,z][c,c']}}"{description}, from=4-6, to=6-6]
	\arrow["{\lambda^{-1}_{[z,z]} \boxtimes 1_{[c,c']}}"{description}, from=5-2, to=7-4]
	\arrow["{\zeta_{zzcc'}}"', from=5-2, to=8-2]
	\arrow["{(\theta(1_{y}) \boxtimes \theta(1_{z})) \boxtimes 1_{[c,c']}}"{description}, from=5-4, to=4-6]
	\arrow["{(1_{I_{\mathcal{V}}} \boxtimes \theta(1_{z})) \boxtimes 1_{[c,c']}}"{description}, from=5-4, to=7-4]
	\arrow["{1_{[y,y]} \boxtimes \zeta_{zzcc'}}", from=6-6, to=8-6]
	\arrow["{[\alpha_{yzc}^{-1} ,\alpha_{yzc'} ]}"{description}, from=6-7, to=10-7]
	\arrow["{[(1_y \otimes f)\circ f',\, g' \circ (1_y \otimes g)]}"{description}, from=6-7, to=12-9]
	\arrow["{(\theta(1_{y}) \boxtimes 1_{[z,z]}) \boxtimes 1_{[c,c']}}"{description}, from=7-4, to=4-6]
	\arrow["{\alpha_{I_{\mathcal{V}}[z,z][c,c']}}"{description}, from=7-4, to=9-4]
	\arrow["{\lambda^{-1}_{[z \otimes c,z \otimes c']}}"{description}, from=8-2, to=11-4]
	\arrow["{[f, g]}"{description}, from=8-2, to=14-2]
	\arrow["{\zeta_{yy(z \otimes c)(z \otimes c')}}"{description}, from=8-6, to=10-7]
	\arrow["{1_{[y,y]} \boxtimes [f,g]}"{description}, from=8-6, to=11-6]
	\arrow["{\theta(1_{y}) \boxtimes 1_{[z,z] \boxtimes[c,c']}}"{description}, from=9-4, to=6-6]
	\arrow["{1_{I_{\mathcal{V}}} \boxtimes \zeta_{zzcc'}}"{description}, from=9-4, to=11-4]
	\arrow["{[1_y \otimes f, 1_y \otimes g]}"', from=10-7, to=12-7]
	\arrow["{\theta(1_{y}) \boxtimes 1_{[z \otimes c,z \otimes c']}}"{description}, from=11-4, to=8-6]
	\arrow["{1_{I_{\mathcal{V}}} \boxtimes [f,g]}"{description}, from=11-4, to=13-4]
	\arrow["{\zeta_{yydd'}}"{description}, from=11-6, to=12-7]
	\arrow["{[f', g']}"{description}, from=12-7, to=12-9]
	\arrow["{\theta(1_{y}) \boxtimes 1_{[d,d']}}"{description}, from=13-4, to=11-6]
	\arrow["{\mathcal{O}_{\mathcal{V}}(f',g')}"{description}, curve={height=30pt}, from=14-2, to=12-9]
	\arrow["{\lambda_{[d,d']}}"{description}, from=14-2, to=13-4]
\end{tikzcd}
\end{adjustbox}
\]
which is more intuitive diagrammatically:
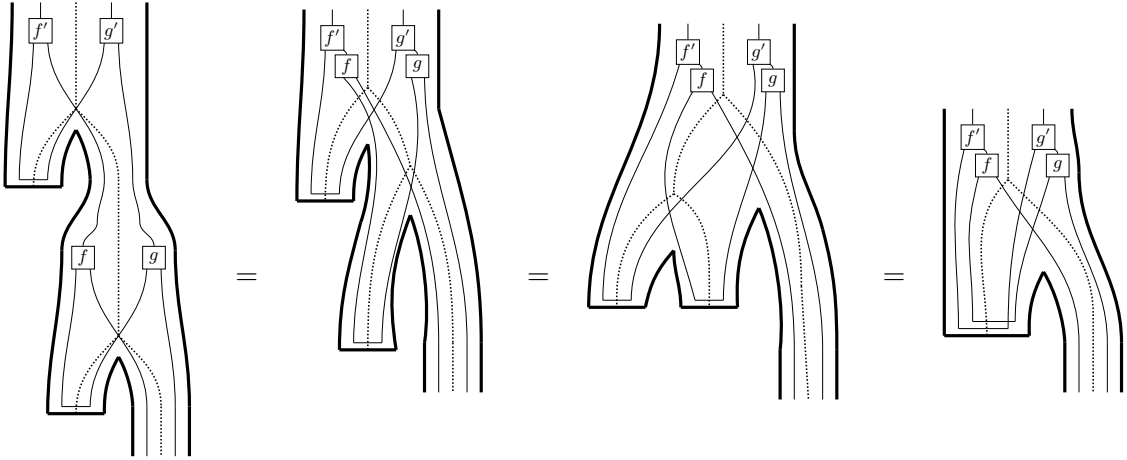
\[\begin{adjustbox}{max width=\textwidth} $\tikzfigscale{0.75}{optic_func_1} \quad = \quad \tikzfigscale{0.75}{optic_func_2} \quad = \quad \tikzfigscale{0.75}{optic_func_3} \quad = \quad \tikzfigscale{0.75}{optic_func_4}.$ \end{adjustbox} \]
To equip this with the structure of a monoidal su-functor, take
\[u^{\mathcal{O}} = \theta(1_{I_{\mathcal{C}}}) := \quad \tikzfigscale{1}{diagram_unit_state}\]
along with $\mu^{\mathcal{O}}_{a,a'} := 1_{[a,a']}$ and $\phi^{\mathcal{O}}_{[c,c'],[d,d']} := \zeta_{cc'dd'}$. The required coherence between $\mu^{\mathcal{O}}$ and $u^{\mathcal{O}}$ trivialises to
\[
\begin{tikzcd}
	{I_{\mathcal{V}}} &  &  & {[I_{\mathcal{C}},I_{\mathcal{C}}]} \\
	{[I_{\mathcal{C}},I_{\mathcal{C}}]} &  &  & {[I_{\mathcal{C}},I_{\mathcal{C}}]}
	\arrow["{\theta(1_{I_{\mathcal{C}}})}", from=1-1, to=1-4]
	\arrow["{u^{\mathcal{O}} = \theta(1_{I_{\mathcal{C}}})}"', from=1-1, to=2-1]
	\arrow["{\mathcal{O}_{\mathcal{V}}(1_{(I_{\mathcal{C}},I_{\mathcal{C}})}) = 1_{[I_{\mathcal{C}},I_{\mathcal{C}}]}}"', from=2-1, to=2-4]
	\arrow["{\mu^{\mathcal{O}}_{I_{\mathcal{C}},I_{\mathcal{C}}} = 1_{[I_{\mathcal{C}},I_{\mathcal{C}}]}}", from=1-4, to=2-4]
\end{tikzcd}
\]
and the tensor coherence trivialises to
\[\begin{tikzcd}
		{[c \otimes d , c' \otimes d']} && {[c, c'] \boxtimes [d,d']} \\
	\\
	{[c \otimes d , c' \otimes d']} && {[c, c'] \boxtimes [d,d']} \\
	\\
	& {[c \otimes d , c' \otimes d']}
	\arrow["{\mu^{\mathcal{O}}_{c \otimes d,c' \otimes d'} = 1}"', from=1-1, to=3-1]
	\arrow["{\zeta_{cc'dd'}}"{description}, from=1-3, to=1-1]
	\arrow["{\mu^{\mathcal{O}}_{c,c'} \boxtimes \mu^{\mathcal{O}}_{d,d'} = 1 \boxtimes 1}", from=1-3, to=3-3]
		\arrow["{\phi^{\mathcal{O}}_{[c,c'],[d,d']} = \zeta_{cc'dd'}}"{description}, from=3-3, to=5-2]
		\arrow["{\mathcal{O}_{\mathcal{V}}(\zeta_{cc'dd'}) = 1_{[c \otimes d , c' \otimes d']}}"{description}, from=5-2, to=3-1]
\end{tikzcd}\]
so $\mathcal{O}_{\mathcal{V}}$ is indeed a monoidal su-functor.

Now let $\mathcal{G}: \mathbf{Optic}[\mathcal{C}] \rightarrow \mathcal{V}$ be any other monoidal su-functor. Since $[a,a'] = \mathcal{O}_{\mathcal{V}}(a,a')$, the natural isomorphisms $\mu^{\mathcal{G}}_{a,a'}:[a,a'] \rightarrow \mathcal{G}(a,a')$ define a comparison $2$-cell (which is furthermore an isomorphism)
\[
\eta:\mathcal{O}_{\mathcal{V}} \Rightarrow \mathcal{G},
\qquad
\eta_{(a,a')} := \mu^{\mathcal{G}}_{a,a'}.
\]
Naturality is the requirement that for every optic $(f,g)$ the diagram
\[\begin{tikzcd}
		{[c,d]} & {\mathcal{G}(c,d)} \\
		{[c',d']} & {\mathcal{G}(c',d')}
	\arrow["{\mathcal{O}_{\mathcal{V}}(f,g)}"', from=1-1, to=2-1]
	\arrow["{\mu^{\mathcal{G}}_{c',d'}}"', from=2-1, to=2-2]
	\arrow["{\mu^{\mathcal{G}}_{c,d}}", from=1-1, to=1-2]
	\arrow["{\mathcal{G}(f,g)}", from=1-2, to=2-2]
\end{tikzcd}\]
commutes, and this may again be verified with string diagrams:
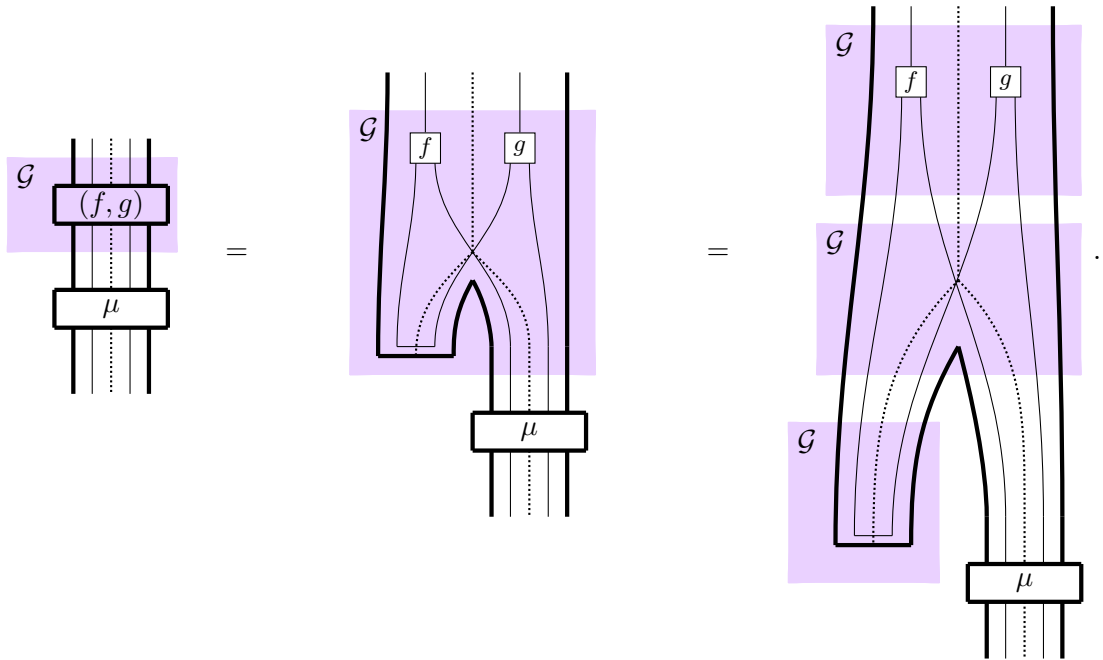
\[ \begin{adjustbox}{max width=\textwidth} $\tikzfigscale{1}{optic_g_0} \quad = \quad \tikzfigscale{1}{optic_g_1} \quad = \quad \tikzfigscale{1}{optic_g_2}.$ \end{adjustbox} \]
then, using su-functoriality of $\mathcal{G}$,
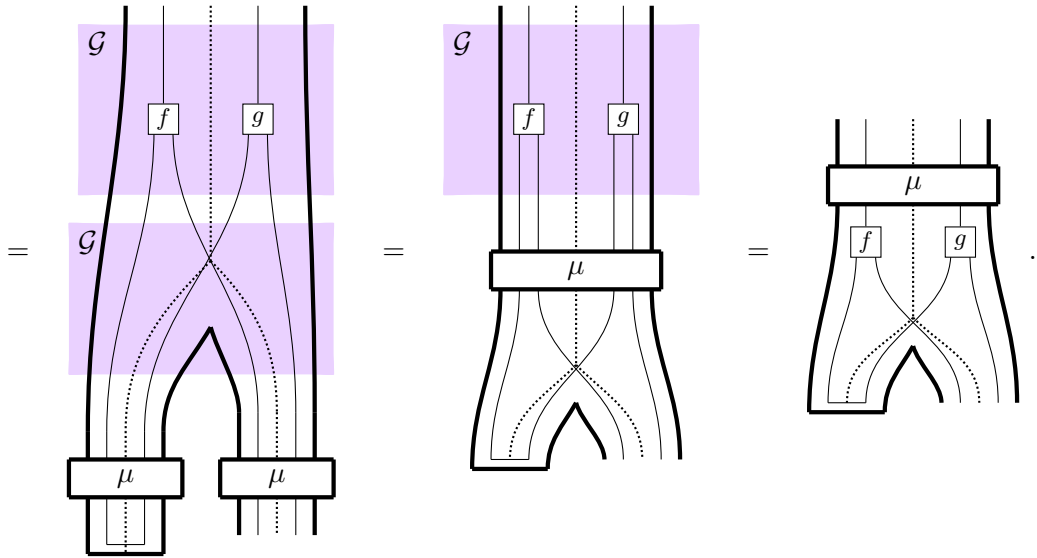
\[ \begin{adjustbox}{max width=\textwidth} $= \quad \tikzfigscale{1}{optic_g_3} \quad = \quad \tikzfigscale{1}{optic_g_4} \quad = \quad \tikzfigscale{1}{optic_g_5}.$ \end{adjustbox} \]
which is precisely the lower-left path of
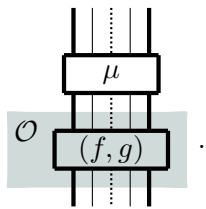
\[
\tikzfigscale{1}{optic_g_6}.
\]
The monoidality of $\eta$ is the commutativity of
\[\begin{tikzcd}
	{[c,c'] \boxtimes [d,d']} && {[c \otimes d,c' \otimes d']} \\
	& {\mathcal{G}((c,c') \otimes (d,d'))} \\
	{\mathcal{G}[c,c'] \boxtimes \mathcal{G}[d,d']} && {\mathcal{G}([c \otimes d, c' \otimes d'])}
	\arrow["{\zeta_{cdc'd'}}", from=1-1, to=1-3]
	\arrow["{\mu^{\mathcal{G}}_{cc'} \boxtimes \mu^{\mathcal{G}}_{dd'}}"', from=1-1, to=3-1]
	\arrow["{\mu^{\mathcal{G}}_{c \otimes d,c' \otimes d'}}", from=1-3, to=3-3]
	\arrow["{\mathcal{G}(1_{(c,c') \otimes (d,d')})}"{description}, from=2-2, to=3-3]
	\arrow["{\phi^{\mathcal{G}}_{[c,c'],[d,d']}}"{description}, from=3-1, to=2-2]
	\arrow["{\phi^{\mathcal{G}}_{[c,c'],[d,d']}}"', from=3-1, to=3-3]
\end{tikzcd}\]
which commutes by su-functoriality of $\mathcal{G}$.
The su-naturality triangle reads directly as
\[\begin{tikzcd}
	& {[c,d]} & \\
	{[c,d]} && {\mathcal{G}(c,d)}
	\arrow["{1_{[c,d]}}"', from=1-2, to=2-1]
	\arrow["{\mu^{\mathcal{G}}_{cd}}", from=1-2, to=2-3]
	\arrow["{\eta_{(c,d)} := \mu^{\mathcal{G}}_{cd}}"', from=2-1, to=2-3]
\end{tikzcd}.\]

For uniqueness, if
\[ \tau:\mathcal{O}_{\mathcal{V}} \Rightarrow \mathcal{G} \]
is any monoidal su-natural transformation, then for every $(a,a')$ one has
\[ \tau_{(a,a')} \circ \mu^{\mathcal{O}}_{a,a'} = \mu^{\mathcal{G}}_{a,a'}. \]
Since $\mu^{\mathcal{O}}_{a,a'} = 1_{[a,a']}$, it follows that
\[ \tau_{(a,a')} = \mu^{\mathcal{G}}_{a,a'} = \eta_{(a,a')}, \]
and so $\eta$ is the unique comparison $2$-cell, and $\mathbf{Optic}[\mathcal{C}]$ is $2$-initial in $\mathbf{MonSuCat}_{\mathcal{C}}$.
\end{proof}
This universal property for optics within theories of supermaps is closely related to the universal property for coend optics arising from the adjunction between (pro)duoidal categories and normal (pro)duoidal categories \cite{earnshaw2023produoidalalgebraprocessdecomposition}. For any symmetric monoidal category $\mathcal{C}$, the category $\mathcal{C}^{op} \times \mathcal{C}$ carries a canonical produoidal structure, and its normalisation under that adjunction is the category of coend optics. The result of \cite{earnshaw2023produoidalalgebraprocessdecomposition} therefore yields an initiality statement for coend optics in the corresponding comma category of produoidal functors out of $\mathcal{C}^{op} \times \mathcal{C}$\footnote{The authors are grateful to James Hefford for identifying, unpacking, and explaining this similarity.}.

Despite the conceptual similarity, the two setups differ in several important ways, suggesting that there may be a more general framework unifying them. On the produoidal side one obtains a laxator of the form $[a,a'] \boxtimes [b,b'] \rightarrow [a \otimes b , a' \otimes b']$, whereas monoidal su-categories are organised around the inverse splitting isomorphism $\mu^{\boxtimes} : [a \otimes b, a' \otimes b'] \rightarrow [a,a'] \boxtimes [b,b']$. Likewise, morphisms in the produoidal comma category preserve the hom-functor $[-,=]$ on the nose, while monoidal su-functors preserve the corresponding structure only up to coherent isomorphism. Finally, monoidal su-categories include the state/process isomorphism $\theta$, while the produoidal perspective keeps track of an additional \textit{sequencing} tensor product which in practice handles multi-input higher-order transformations. At a high level, monoidal su-categories focus on the completeness of higher-order transformations, whereas produoidal categories focus on their multi-hole aspect. Taken together, these perspectives suggest that in a categorical semantics for multi-hole higher-order transformations with completeness included, one could still reasonably hope to frame coend optics as suitably initial.



\section{Conclusion and Outlook}
Monoidal su-categories provide a categorical algebra of single-hole higher-order processes. The framework isolates the data needed to interpret a hole object $[a,a']$, the state/process correspondence $\theta$, and the splitting map $\mu^{\boxtimes}$ that makes local application to bipartite processes possible. Over a fixed symmetric monoidal base $\mathcal{C}$, these structures assemble into the $2$-category $\mathbf{MonSuCat}_{\mathcal{C}}$.


Several directions now seem especially promising:
\begin{itemize}
\item \textit{Coherence for monoidal su-categories.} Monoidal su-categories involve two interacting families of coherence, and we conjecture that every monoidal su-category is equivalent to one in which both the underlying category $\mathcal{C}$ and the higher-order category $\mathcal{V}$ are strictified along with the isomorphisms $\theta$ and $\mu$. Stating such a theorem will require extending the notion of monoidal su-functor to allow both the underlying and higher-order category to vary.
\item \textit{An expressive graphical language for higher-order quantum operations.} The ZX calculus equips traditional monoidal categories with gadgets (representing a core feature of quantum mechanics, the existence of complementary observables), that allow for expressive and automated graphical reasoning about correctness and compilation in first-order quantum protocols \cite{Duncan2019Graph-theoreticZX-calculus, Coecke2017PicturingReasoning}. By establishing a suitable replacement for the traditional circuit model, monoidal su-categories give a natural starting point for building an expressive graphical language for reasoning with higher-order quantum protocols.
\item \textit{Dynamical resource theories.} The notion of monoidal subcategory is used to formulate in categorical terms what it means to be a resource theory \cite{Coecke2014AResources}, and so the notion of monoidal su-subcategory suggests a clean categorical approach to dynamical resource theories \cite{GourDynamicalResources}, where the resources to be transformed are themselves processes.
\item \textit{Beyond the single-hole setting.} The next natural step is to extend the framework from single-hole to multi-input higher-order structure. The enrichment-based approach of \cite{wilson2023mathematical}, the polycategorical approach of \cite{wilson2022freepolycategoriesunitarysupermaps}, and recent logical/categorical work on duoidal and BV-style structure \cite{COCKETT1997133,SimmonsKissinger2022,hefford2025bvcategoryspacetimeinterventions, hirata_et_al:LIPIcs.LICS.2026.57, earnshaw2023produoidalalgebraprocessdecomposition} will inform further development towards a broader non-strict categorical semantics of multi-input higher-order processes.
\end{itemize}
Most broadly, the definition of monoidal su-categories captures succinctly a simple compositional idea in formal categorical algebra, providing a stable foundation on which more domain-specific tools can be built and through which cross-domain pollination between formal computer science and the foundations of quantum information theory can be facilitated.

\section{Acknowledgements}
The authors are grateful to James Hefford for useful discussions, particularly regarding the relationship between the universal property for optics established here and the adjunction between produoidal and normal-produoidal categories. 
Part of MW's contribution was carried out at UCL and part at CentraleSupélec. While based at University College London, MW was funded by the Engineering and Physical Sciences Research Council [grant number EP/W524335/1].

\bibliography{ref_local}

\appendix

\section{Detailed Proofs and Commutative Diagrams}
The main text uses string diagrams where they are more readable. Here we record the full proofs in terms of commutative diagrams where necessary for completeness.

\subsection[The 2-category MonSuCat over C]{The $2$-category $\mathbf{MonSuCat}_{\mathcal{C}}$}\label{app:monsucat-2cat}
We now give the full proof of Theorem~\ref{thm:monsucat-2cat}. Monoidal categories, together with monoidal functors and monoidal natural transformations, form the $2$-category $\mathbf{MonCat}$. Adding su-structure, one must check that the additional $\mu$-morphisms are closed under composition and that the corresponding coherence is preserved by units, vertical composition, and whiskering.

For closure under composition, suppose
\[
(\mathcal{F},\mu^{\mathcal{F}}):(\mathcal{V},\mathcal{C}) \rightarrow (\mathcal{V}',\mathcal{C})
\qquad\text{and}\qquad
(\mathcal{G},\mu^{\mathcal{G}}):(\mathcal{V}',\mathcal{C}) \rightarrow (\mathcal{V}'',\mathcal{C})
\]
are monoidal su-functors. We now check the following two diagrams:
\[\begin{tikzcd}
	{I_{\mathcal{V}''}} &&& {[I_{\mathcal{C}},I_{\mathcal{C}}]} \\
	& {\mathcal{G}(I_{\mathcal{V}'})} && {\mathcal{G}[I_{\mathcal{C}},I_{\mathcal{C}}]} \\
	{\mathcal{G}\mathcal{F}(I_{\mathcal{V}})} &&& {\mathcal{G}\mathcal{F}[I_{\mathcal{C}},I_{\mathcal{C}}]}
	\arrow["{\theta''(1_{I_{\mathcal{C}}})}", from=1-1, to=1-4]
	\arrow["{u^{\mathcal{G}}}"', from=1-1, to=2-2]
	\arrow["{u^{\mathcal{G}\mathcal{F}}}"', from=1-1, to=3-1]
	\arrow["{\mu^{\mathcal{G}}_{I_{\mathcal{C}},I_{\mathcal{C}}}}", from=1-4, to=2-4]
	\arrow["{\mathcal{G}(\theta'(1_{I_{\mathcal{C}}}))}", from=2-2, to=2-4]
	\arrow["{\mathcal{G}(u^{\mathcal{F}})}", from=2-2, to=3-1]
	\arrow["{\mathcal{G}(\mu^{\mathcal{F}}_{I_{\mathcal{C}},I_{\mathcal{C}}})}", from=2-4, to=3-4]
	\arrow["{\mathcal{G}\mathcal{F}(\theta(1_{I_{\mathcal{C}}}))}"', from=3-1, to=3-4]
\end{tikzcd},\]
and
\[
\begin{adjustbox}{max width=\textwidth}
\begin{tikzcd}
		{[c \otimes d , c' \otimes d']} &&& {[c, c'] \boxtimes [d,d']} \\
		& {\mathcal{G}[c \otimes d , c' \otimes d']} \\
		&& {\mathcal{G}[c, c'] \boxtimes \mathcal{G}[d,d']} \\
	{\mathcal{G}\mathcal{F}[c \otimes d , c' \otimes d']} & {\mathcal{G}([c, c'] \boxtimes [d,d'])} \\
	\\
	& {\mathcal{G}(\mathcal{F}[c, c'] \boxtimes \mathcal{F}[d,d'])} \\
	&& {\mathcal{G}\mathcal{F}[c, c'] \boxtimes \mathcal{G}\mathcal{F}[d,d']} \\
	& {\mathcal{G}\mathcal{F}([c, c'] \boxtimes [d,d'])}
	\arrow["{\mu^{\mathcal{G}}_{c \otimes d,c' \otimes d'}}"{description}, from=1-1, to=2-2]
	\arrow["{\mu^{\mathcal{G}\mathcal{F}}_{c \otimes d,c' \otimes d'}}"{description}, from=1-1, to=4-1]
	\arrow["{\zeta_{cc'dd'}}"{description}, from=1-4, to=1-1]
	\arrow["{\mu^{\mathcal{G}}_{c,c'} \boxtimes \mu^{\mathcal{G}}_{d,d'}}"{description}, from=1-4, to=3-3]
	\arrow["{\mu^{\mathcal{G}\mathcal{F}}_{c,c'} \boxtimes \mu^{\mathcal{G}\mathcal{F}}_{d,d'}}"{description}, curve={height=-30pt}, from=1-4, to=7-3]
	\arrow["{\mathcal{G}(\mu^{\mathcal{F}}_{c \otimes d,c' \otimes d'})}"{description}, from=2-2, to=4-1]
	\arrow["{\phi^{\mathcal{G}}_{[c,c'],[d,d']}}"{description}, from=3-3, to=4-2]
	\arrow["{\mathcal{G}(\mu^{\mathcal{F}}_{c,c'}) \boxtimes \mathcal{G}(\mu^{\mathcal{F}}_{d,d'})}"{description}, from=3-3, to=7-3]
	\arrow["{\mathcal{G}(\zeta_{cc'dd'})}"{description}, from=4-2, to=2-2]
	\arrow["{\mathcal{G}(\mu^{\mathcal{F}}_{c,c'} \boxtimes \mu^{\mathcal{F}}_{d,d'})}"{description}, from=4-2, to=6-2]
	\arrow["{\mathcal{G}(\phi^{\mathcal{F}}_{[c,c'],[d,d']})}"{description}, from=6-2, to=8-2]
	\arrow["{\phi^{\mathcal{G}}_{\mathcal{F}[c,c'],\mathcal{F}[d,d']}}"{description}, from=7-3, to=6-2]
		\arrow["{\phi^{\mathcal{G}\mathcal{F}}_{[c,c'],[d,d']}}"{description}, curve={height=-24pt}, from=7-3, to=8-2]
		\arrow["{\mathcal{G}\mathcal{F}(\zeta_{cc'dd'})}"{description}, curve={height=-30pt}, from=8-2, to=4-1]
\end{tikzcd}.
\end{adjustbox}
\]

For associativity of composition, note that
\[
\mu^{(\mathcal{H}\mathcal{G})\mathcal{F}}
=
\mathcal{H}\mathcal{G}(\mu^{\mathcal{F}})\circ \mu^{\mathcal{H}\mathcal{G}}
=
\mathcal{H}\mathcal{G}(\mu^{\mathcal{F}})\circ \mathcal{H}(\mu^{\mathcal{G}})\circ \mu^{\mathcal{H}}
=
\mathcal{H}(\mu^{\mathcal{G}\mathcal{F}})\circ \mu^{\mathcal{H}}
=
\mu^{\mathcal{H}(\mathcal{G}\mathcal{F})}.
\]
For the unit $1$-cell, one simply takes the identity functor equipped with the identity natural transformation $\mu^{\mathcal{I}}:[-,=]\Rightarrow [-,=]$, with the extra monoidal su-functor coherences satisfied immediately. For unitality of sequential composition, it is immediate that $\mu^{\mathcal{G}\mathcal{I}} = \mu^{\mathcal{G}}$ and $\mu^{\mathcal{I}\mathcal{F}} = \mu^{\mathcal{F}}$.

The identity natural transformation satisfies the required su-coherence by commutativity of
\[\begin{tikzcd}
	& {[c,c']} & \\
	{\mathcal{F}[c,c']} && {\mathcal{F}[c,c']}
	\arrow["{\mu^{\mathcal{F}}_{c,c'}}"', from=1-2, to=2-1]
	\arrow["{\mu^{\mathcal{F}}_{c,c'}}", from=1-2, to=2-3]
	\arrow["{1_{[c,c']}}"', from=2-1, to=2-3]
\end{tikzcd}.\]

Finally, one checks that the extra coherences required for monoidal su-natural transformations are preserved under sequential composition and whiskering. For sequential composition $\tau \circ \eta$ of natural transformations $\eta : \mathcal{F} \Rightarrow \mathcal{G}$ and $\tau : \mathcal{G} \Rightarrow \mathcal{H}$ we have commutativity of
\[\begin{tikzcd}
	&& {[c,c']} && \\
	{\mathcal{F}[c,c']} && {\mathcal{G}[c,c']} && {\mathcal{H}[c,c']}
	\arrow["{\mu^{\mathcal{F}}_{c,c'}}"', from=1-3, to=2-1]
	\arrow["{\mu^{\mathcal{G}}_{c,c'}}", from=1-3, to=2-3]
	\arrow["{\mu^{\mathcal{H}}_{c,c'}}", from=1-3, to=2-5]
	\arrow["{\eta_{[c,c']}}"', from=2-1, to=2-3]
	\arrow["{\tau_{[c,c']}}"', from=2-3, to=2-5]
\end{tikzcd},\]
whereas for whiskering $\eta \star \tau$ with $\eta : \mathcal{F} \Rightarrow \mathcal{F}'$ and $\tau : \mathcal{G} \Rightarrow \mathcal{G}'$ one may verify coherence for the composite by observing commutativity of
\[\begin{tikzcd}
	&&& {[c,c']} &&& \\
	\\
	&&& {\mathcal{G}[c,c']} \\
	&&&& {\mathcal{G}'[c,c']} \\
	&&& {\mathcal{G}\mathcal{F}'[c,c']} \\
	{\mathcal{G}\mathcal{F}[c,c']} &&&&&& {\mathcal{G}'\mathcal{F}'[c,c']}
	\arrow["{\mu^{\mathcal{G}}_{c,c'}}"', from=1-4, to=3-4]
	\arrow["{\mu^{\mathcal{G}\mathcal{F}}_{c,c'}}"', curve={height=30pt}, from=1-4, to=6-1]
	\arrow["{\mu^{\mathcal{G}'\mathcal{F}'}_{c,c'}}", curve={height=-30pt}, from=1-4, to=6-7]
	\arrow["{\tau_{[c,c']}}", from=3-4, to=4-5]
	\arrow["{\mathcal{G}(\mu^{\mathcal{F}'}_{c,c'})}"', from=3-4, to=5-4]
	\arrow["{\mathcal{G}(\mu^{\mathcal{F}}_{c,c'})}"', from=3-4, to=6-1]
	\arrow["{\mathcal{G}'(\mu^{\mathcal{F}'}_{c,c'})}"{description}, from=4-5, to=6-7]
	\arrow["{\tau_{\mathcal{F}'[c,c']}}", curve={height=18pt}, from=5-4, to=6-7]
	\arrow["{\mathcal{G}(\eta_{[c,c']})}", curve={height=18pt}, from=6-1, to=5-4]
	\arrow["{(\tau \star \eta)_{[c,c']}}"', curve={height=30pt}, from=6-1, to=6-7]
\end{tikzcd}\]


\section{Examples of Monoidal Su-Categories}

\begin{example}[Combs]\label{ex:combs}
Following \cite{Chiribella2009TheoreticalNetworks,hefford_coend}, the symmetric monoidal category $\mathbf{Comb}[\mathcal{C}]$ has the same objects as $\mathbf{Optic}[\mathcal{C}]$, namely pairs $(a,a')$, and morphisms
\[\mathbf{Comb}[\mathcal{C}]((a,a'),(b,b')):=
\mathbf{Optic}[\mathcal{C}]((a,a'),(b,b'))/\!\cong_{\mathrm{comb}},
\]
where two representatives $(f_1,g_1)$ and $(f_2,g_2)$ are considered equivalent when, for every pair of objects $x,x' \in \mathcal{C}$ and every morphism $\phi:a \otimes x \rightarrow a' \otimes x'$, the induced composites
\[b \otimes x \xrightarrow{f_1 \otimes 1_x} (m_1 \otimes a) \otimes x \cong m_1 \otimes (a \otimes x) \xrightarrow{1_m \otimes \phi} m_1 \otimes (a' \otimes x') \cong (m_1 \otimes a') \otimes x' \xrightarrow{g_1 \otimes 1_{x'}} b' \otimes x' \]
and the corresponding composite built from $(f_2,g_2)$ agree.
We can express this requirement graphically by
\[  \tikzfigscale{1}{optic_comp_1}   \  \cong_{\mathbf{Comb}}  \   \tikzfigscale{1}{optic_comp_2} \ \iff  \ \forall x,x', \phi : a \otimes x \rightarrow a' \otimes x' :  \  \tikzfigscale{1}{comb_def_1} \ = \  \tikzfigscale{1}{comb_def_2} .  \]
 Sequential and parallel composition are inherited directly from the corresponding composition rules in the category of optics.
\end{example}

For our next example we refer to a more abstract approach to defining supermaps, namely as locally applicable transformations \cite{wilson_locality,wilson2026supermapsgeneralisedtheories}, which are in fact morphisms between strong profunctors \cite{hefford_supermaps}.
We say that a locally-applicable transformation of type
\[S:(a,a') \longrightarrow (b,b')\]
is given by families of functions
\[S_{x,x'}:\mathcal{C}(a \otimes x,a' \otimes x') \longrightarrow \mathcal{C}(b \otimes x,b' \otimes x')\]
indexed by pairs $x,x' \in \mathcal{C}$ and natural over the category of coend-optics, meaning commuting with combs on the environment.
Graphically, we may denote a locally applicable transformation by
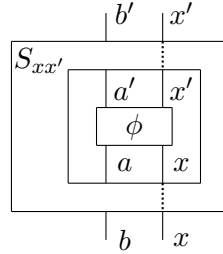
\[  \tikzfigscale{1}{slot_def_1}   \] 
with naturality condition expressed by
\[   \ \forall x,x', \phi, f,g : a \otimes x \rightarrow a' \otimes x' :  \  \tikzfigscale{1}{slot_def_2} \ = \  \tikzfigscale{1}{slot_def_3} .  \]
\begin{example}[Slots]\label{ex:slots}
Locally applicable transformations do not in general form monoidal categories, because the interchange law can fail \cite{wilson2022freepolycategoriesunitarysupermaps}. Slots are defined as locally applicable transformations satisfying a stronger locality condition, namely commutativity with any locally applicable transformation acting on the environment. Graphically, this condition can be expressed by
\[  \tikzfigscale{1}{slot_def_4} \ = \  \tikzfigscale{1}{slot_def_5} .  \]
The symmetric monoidal category $\mathbf{Slot}[\mathcal{C}]$ can be equipped with the structure of a symmetric monoidal su-category as follows.
\begin{itemize}
    \item The functor $[-,=]_{\cong}:\mathcal{C}^{op} \times \mathcal{C} \rightarrow \mathbf{Slot}[\mathcal{C}]$ is defined on objects by $[a,a'] = (a,a')$ and sends a pair $(f:b \rightarrow a,g:a' \rightarrow b')$ to the locally applicable transformation
    \[
    [f,g]_{x,x'}(\phi):=(g \otimes 1_{x'}) \circ \phi \circ (f \otimes 1_x).
    \]
    \item The natural isomorphism $\mathbf{Slot}[\mathcal{C}]((I_{\mathcal{C}},I_{\mathcal{C}}),(c,c')) \cong \mathcal{C}(c,c')$ arises by a Yoneda-like argument. Concretely, note that any locally applicable transformation of type $(I,I) \rightarrow (c,c')$ has the form $S_{xx'}(\phi) = S_{II}(1_{I}) \otimes \phi$ for some fixed morphism $S_{II}(1_I)$. This gives a way to construct a representative morphism for each locally applicable transformation. Conversely, given any morphism $f : c \rightarrow c'$, one may construct an associated natural transformation $S_{xx'}(\phi) := f \otimes \phi$. It is routine to check that these two constructions are mutually inverse to each other and natural in $c,c'$.
    \item The definition of the monoidal product in the category of slots is defined to be $(c,c') \otimes (d,d') := (c \otimes d, c' \otimes d')$.
\end{itemize}
The coherence axioms for monoidal su-categories are then checked directly.
\end{example}


\end{document}